\documentclass{amsart}

\usepackage{amsmath,amsthm,amssymb,bm}
\usepackage{hyperref}
\usepackage{a4wide}
\usepackage{cleveref}
\usepackage{graphicx,color}
\usepackage{ytableau}
\usepackage{tikz}
\usepackage{tikz-cd}
\usepackage{subcaption}
\numberwithin{equation}{section}

\newtheorem{thm}{Theorem}[section]
\newtheorem{lem}[thm]{Lemma}
\newtheorem{prop}[thm]{Proposition}
\newtheorem{cor}[thm]{Corollary}
\theoremstyle{definition}
\newtheorem{exam}[thm]{Example}
\newtheorem{defn}[thm]{Definition}
\newtheorem{conj}[thm]{Conjecture}

\newtheorem{remark}[thm]{Remark}

\crefname{lem}{Lemma}{Lemmas}
\crefname{thm}{Theorem}{Theorems}
\crefname{prop}{Proposition}{Propositions}
\crefname{question}{Question}{Questions}
\crefname{defn}{Definition}{Definitions}
\crefname{conj}{Conjecture}{Conjectures}
\crefname{figure}{Figure}{Figures}
\crefname{cor}{Corollary}{Corollaries} 
\crefformat{equation}{(#2#1#3)}
\Crefformat{equation}{Equation #2(#1)#3}

\AddToHook{env/lem/begin}{\crefalias{thm}{lem}}
\AddToHook{env/prop/begin}{\crefalias{thm}{prop}}
\AddToHook{env/cor/begin}{\crefalias{thm}{cor}}
\AddToHook{env/exam/begin}{\crefalias{thm}{exam}}
\AddToHook{env/defn/begin}{\crefalias{thm}{defn}}
\AddToHook{env/conj/begin}{\crefalias{thm}{conj}}
\AddToHook{env/question/begin}{\crefalias{thm}{question}}
\AddToHook{env/problem/begin}{\crefalias{thm}{problem}}
\AddToHook{env/remark/begin}{\crefalias{thm}{remark}}

\newcommand\LHT{\operatorname{LHT}}
\newcommand\RST{\operatorname{RST}}

\newcommand\Par{\operatorname{Par}}

\newcommand\QQ{\mathbb{Q}}

\newcommand\LL{\mathcal{L}}

\DeclareMathOperator{\wt}{wt}

\newcommand\x{\vec x}

\def\multinom#1#2{\ensuremath{\left(\kern-.4em\left(\genfrac{}{}{0pt}{}{#1}{#2}\right)\kern-.4em\right)}}

\renewcommand\vec[1]{\mathbf{#1}}
\newcommand\norm[1]{\lVert #1 \rVert}

\newcommand\qHyper[5]{{}_{#1}\phi_{#2} \left(
    \begin{matrix}
      #3\\
      #4\\
    \end{matrix}
    ; #5
    \right)}

\title{An explicit formula for Koornwinder moments and Rains' positivity conjecture}

\author{Younggwang Cho}
\address{Department of Mathematics, Sungkyunkwan University, Suwon, South Korea}
\email{brglory@g.skku.edu}

\author{Donghyun Kim}
\address{Department of Mathematics, Ewha Womans University, Seoul, South Korea}
\email{kdh310@ewha.ac.kr}

\author{Jang Soo Kim}
\address{Department of Mathematics, Sungkyunkwan University, Suwon, South Korea}
\email{jangsookim@skku.edu}

\author{Hojoon Lee}
\address{Department of Mathematics, Sungkyunkwan University, Suwon, South Korea}
\email{hojoon1101@g.skku.edu}

\author{Jing Liu}
\address{Research Center for Mathematics and Interdisciplinary Sciences,
Shandong University, Qingdao, P.R. China}
\email{lsweet@mail.sdu.edu.cn}

\author{Minho Song}
\address{Department of Mathematics, Yonsei University, Seoul, South Korea}
\email{minhosong@yonsei.ac.kr}

\thanks{D.~H.~Kim was supported by the Ewha Womans University Research Grant of 2026.
J.~S.~Kim was supported by the National Research Foundation of Korea (NRF) grant funded by the Korea government (RS-2025-00557835).
J.~Liu was supported by the China Scholarship Council (No.~202506220083) during her visit to Sungkyunkwan University.
M.~Song was supported by the National Research Foundation of Korea (NRF) grant funded by the Korea government (MSIT) (No.~2022R1C1C2009025).}


\begin{document}

\begin{abstract}
The asymmetric simple exclusion process (ASEP) is an important particle model with deep connections to orthogonal polynomials. Motivated by this connection, Corteel and Williams introduced the Koornwinder moments $M^{Z}_{\lambda}$ at $ t=q $, which generalize the moments of Askey--Wilson polynomials. They showed that the partition function of the two-species ASEP is equal to $M^{Z}_{\lambda}$ for a one-row partition $ \lambda $.

In this paper, we investigate a conjecture of Rains on the positivity of the minimal numerator of the Koornwinder moment $M^{Z}_{\lambda}$. We derive the first explicit formula for this moment, thereby obtaining a precise formulation of the conjecture by determining the minimal denominator of $M^{Z}_{\lambda}$. We also propose a generalization of the conjecture for the more general Koornwinder moments $M^{Z}_{\lambda,\mu}$ indexed by two partitions at special parameter values.

We prove the generalized Rains' conjecture in two special cases: $(\xi,q)=(1,0)$ and $(\xi,q)=(1,1)$. For $(\xi,q)=(1,0)$, we construct a lattice path model and obtain a combinatorial formula for $M^{Z}_{\lambda,\mu}$ in terms of non-intersecting lattice paths. For $(\xi,q)=(1,1)$, we establish an explicit product formula for $M^{Z}_{\lambda,\mu}$ and give a combinatorial interpretation using lecture hall tableaux.
\end{abstract}

\maketitle


\section{Introduction}

\subsection{ASEP and orthogonal polynomials}
The asymmetric simple exclusion process (ASEP) is an important model in statistical mechanics that describes a system of particles on a lattice hopping to the left and right. This process was introduced independently around 1970 in both biological and mathematical contexts \cite{MacdonaldGibbsPipkin1968, Spitzer1970}. Since then, the model has been studied extensively across various fields for a number of reasons. First, ASEP exhibits rich phenomenology and has numerous applications, including protein synthesis, traffic flow, shock formation, surface growth, and sequence alignment (see \cite{ChouMallickZia2011} for a summary). Second, the study of ASEP admits a wide range of mathematical approaches, such as the Bethe Ansatz \cite{GwaSpohn1992}, quadratic algebras \cite{EsslerRittenberg1996}, combinatorics \cite{CMW2018, FM07, KW23}, orthogonal polynomials \cite{CMW, Corteel2007a, Corteel2019, Corteel2011}, and random matrices \cite{Johansson2000}.

One of the most remarkable aspects of ASEP (on a line) is its connection to orthogonal polynomials. Orthogonal polynomials attracted the attention of combinatorialists when mathematicians such as Flajolet \cite{Flajolet1980}, Viennot \cite{ViennotOP}, and Foata \cite{Foata1978} initiated a combinatorial study of these polynomials. More precisely, the moments of many well-known orthogonal polynomials exhibit rich combinatorial structures, and there has been extensive work uncovering the combinatorics of these moments for orthogonal polynomials in the Askey scheme \cite{dSW1995, ISV1987, KSZ2011}.

In particular, the stationary distribution of ASEP is closely related to the theory of Askey--Wilson polynomials, which lie at the top of the hierarchy of basic hypergeometric orthogonal polynomials. This connection was first revealed by Sasamoto \cite{Sasamoto1999} and subsequently by Uchiyama, Sasamoto, and Wadati \cite{USW2004}, where the partition function of ASEP was linked to moments of the Askey--Wilson polynomials. Such an unexpected discovery sparked an explosion of research, which eventually led to explicit combinatorial formulas for the stationary distribution of ASEP. Corteel and Williams \cite{Corteel2011} introduced staircase tableaux and provided a combinatorial formula for the stationary distribution of ASEP in terms of these tableaux, thereby linking ASEP with the rich combinatorics underlying Askey--Wilson polynomials.

The Askey--Wilson polynomial is a special case of the Koornwinder polynomial. The Koornwinder polynomial is a multivariate orthogonal polynomial, often called the type \(BC\) Macdonald polynomial \cite{Koornwinder1992}. Since the Askey--Wilson polynomial is connected with ASEP, it is natural to ask whether there exists a general version of ASEP that is related to Koornwinder polynomials. Cantini \cite{Cantini2017} answered this question by showing that the partition function of the multi-species ASEP is given as a specialization of Koornwinder polynomials. However, it remains a wide-open problem to reveal a combinatorial structure of the stationary distribution of the multi-species ASEP (and thereby a combinatorial structure of Koornwinder polynomials). The best progress in this direction is a formula given by Corteel, Mandelshtam, and Williams \cite{CMW} for the two-species ASEP in terms of rhombic staircase tableaux.

Corteel and Williams \cite{Corteel2019} introduced the Koornwinder moments \(M^{Z}_{\lambda} = M^{Z}_{\lambda}(\xi; \alpha, \beta, \gamma, \delta; q)\) as the multivariate moments of rescaled Askey--Wilson polynomials, which can be viewed as a specialization of Koornwinder polynomials. They also showed that the Koornwinder moment for a one-row partition is the partition function for the two-species ASEP. We caution the reader that the Koornwinder moment, in general, is not the partition function for the multi-species ASEP, since it arises from a different specialization than the one considered by Cantini \cite{Cantini2017}. These two specializations coincide only for \(M^{Z}_{\lambda}\) when \(\lambda\) is a one-row partition, which is connected to the two-species ASEP. It is also worth mentioning that Corteel et al. \cite{Corteel2024} further constructed Koornwinder polynomials for one-column partitions, building upon the combinatorics of rhombic staircase tableaux.

Throughout this paper, we follow the corrected version \cite{Corteel2019} of the
published paper \cite{Corteel2018} by Corteel and Williams.

\subsection{Main Results}

In this paper, we investigate Rains' conjecture \cite[Conjecture
4.4]{Corteel2019}, which states that the minimal numerator of the Koornwinder
moment \(M^{Z}_{\lambda}\) is a polynomial in
\( \xi, \alpha, \beta, \gamma, \delta, q \) with nonnegative integer
coefficients. This conjecture was verified in a special case
when \(\lambda\) is a one-row partition, where a combinatorial formula
has been established in terms of rhombic staircase tableaux
\cite{CMW}. Our main results are as follows:
\begin{enumerate}
\item A precise formulation of Rains' conjecture, giving an
  explicit formula for the minimal denominator of
  \( M^{Z}_{\lambda} \) (\Cref{prop:den}). We also generalize Rains'
  conjecture to more general Koornwinder moments
  \( M^{Z}_{\lambda, \mu} \) when \( q = 1 \),
  or when one of \( \alpha,\beta,\gamma,\delta,q \) is zero (\Cref{con:gen_rains}).
\item The first explicit formula for the Koornwinder moment \(M^{Z}_{\lambda}\) (\Cref{cor:explicit_Koornwinder_moment}). 
\item A proof of the generalized Rains' conjecture for two special
  cases: \((\xi,q)=(1,0)\) and \((\xi,q)=(1,1)\).
  For the case
  \((\xi,q)=(1,0)\), we construct a weighted lattice model, thereby
  giving an explicit combinatorial formula for the Koornwinder moment
  \( M^{Z}_{\lambda,\mu} \) in terms of non-intersecting lattice paths.
  For the case \((\xi,q)=(1,1)\), we present an explicit product
  formula for \( M^{Z}_{\lambda,\mu} \), and provide a combinatorial model
  using lecture hall tableaux.
\end{enumerate}

The Koornwinder moment is defined as the determinant of a matrix whose
entries are (rescaled) Askey--Wilson mixed moments. The mixed moments
and coefficients of Askey--Wilson polynomials can be realized, up to
sign, as special cases of Koornwinder moments corresponding to one-row
and one-column partitions, respectively. While mixed moments for various
orthogonal polynomials have been extensively studied, the coefficients
of orthogonal polynomials have received much less attention.
Combinatorial formulas for coefficients of various orthogonal
polynomials will be presented in a sequel paper.

\subsection{Organization}
This paper is organized as follows. In \Cref{sec: pre}, we introduce the necessary background on orthogonal polynomials and the two-species ASEP. In \Cref{sec:askey-wilson-moments}, we introduce rescaled Askey--Wilson polynomials, which are the main objects of study in this paper. We then discuss previous results connecting the (mixed) moments of rescaled Askey--Wilson polynomials to the (two-species) ASEP. We also describe the three-term recurrence coefficients of the rescaled Askey--Wilson polynomials in terms of the ASEP parameters $\alpha,\beta,\gamma,\delta$. In \Cref{sec:koornw-moments}, we introduce Koornwinder moments and Rains' conjecture, and prove the first explicit formula for the Koornwinder moments. As a byproduct, we describe the minimal denominator for the Koornwinder moment, thereby stating Rains' conjecture in a complete form. In \Cref{sec:proof-special-cases}, we establish positivity results for certain numerators of the Koornwinder moments in two special cases. Finally, in \Cref{sec:minimal-denominators}, we use these results to prove the generalized Rains' conjecture in these cases and complete the argument for the minimal denominator.

\section{Preliminaries}\label{sec: pre}

This section introduces the necessary definitions and preliminary results on partitions, orthogonal polynomials, and the two-species ASEP.

\subsection{Basic definitions}
\label{subsec:basic_def}
A \emph{partition} is a weakly decreasing sequence
\(\lambda = (\lambda_1, \dots, \lambda_n) \) of positive integers.
Each entry \(\lambda_i\) is called a \emph{part} of \(\lambda\), and
the number of parts is the \emph{length} of \(\lambda\), denoted by
\(\ell(\lambda)\). We use the convention \( \lambda_i=0 \) for
\( i>\ell(\lambda) \). Let \(\Par_n\) denote the set of partitions
with at most \(n\) parts. When we write \(\lambda \in \Par_n\), we
regard \(\lambda\) as a sequence
\( (\lambda_1,\dots,\lambda_n) \) of \( n \) integers by appending
\(n - \ell(\lambda)\) trailing zeros. We also use the notation \((a^b)\)
for the sequence consisting of \(b\) copies of \(a\).
For example, \( (4,3^3,0^2)=(4,3,3,3,0,0)\in \Par_6 \).

The \emph{Young diagram} of a partition \(\lambda\) is
the set
\[
\{ (i, j) \in \mathbb{Z}^2 : 1 \le i \le \ell(\lambda), \; 1 \le j \le \lambda_i \}.
\]
Each element \( (i,j) \) in the Young diagram of \( \lambda \) is
called a \emph{cell}. We will identify a partition with its Young
diagram. The \emph{conjugate} of \(\lambda\), denoted by \(\lambda'\),
is the partition whose Young diagram is
\[
  \{ (j, i) \in \mathbb{Z}^2 : 1 \le i \le \ell(\lambda), \; 1 \le j
  \le \lambda_i \}.
\]
For a cell \((i,j)\) in a partition \(\lambda\), the \emph{hook
  length} \(h(i,j)\) and the \emph{content} \(c(i,j)\) are defined by
\[
h(i,j) = \lambda_i + \lambda'_j - i - j + 1, \qquad c(i,j) = j - i.
\]

For two partitions \(\lambda\) and \(\mu\), we write
\(\mu \subseteq \lambda\) when the Young diagram of \(\mu\) is
contained in that of \(\lambda\). In this case, a \emph{skew
  partition} \(\lambda / \mu\) is given by the set difference
\( \lambda-\mu \) of the Young diagrams. The number of cells in
\(\lambda / \mu\) is denoted by \(|\lambda / \mu|\).

We will use the $q$-Pochhammer symbols:
\[
    (a;q)_n := \prod_{k=0}^{n-1}(1-aq^k), \qquad (a_1, \dots, a_m; q)_{n} := (a_1; q)_{n} \cdots (a_m; q)_{n},
\]
where \( n \) may be \( \infty \), in which case the product is over all \( k\ge0 \).

\subsection{Orthogonal polynomials}

In this subsection, we define orthogonal polynomials and their mixed
moments and review their basic properties. For later use, we state the
connection between the mixed moments of orthogonal polynomials and
their rescaled versions.

\begin{defn}
  Let \(\LL\) be a linear functional defined on the space of polynomials
  in \(x\). A sequence of polynomials \(\{P_n(x)\}_{n \ge 0}\) is called an
  \emph{orthogonal polynomial sequence} (OPS) with respect to \(\LL\) if
  the following conditions hold:
\begin{itemize}
    \item \(\deg P_n(x) = n\) for all \(n \ge 0\);
    \item 
    \(
    \LL\bigl(P_n(x) P_m(x)\bigr) = K_n \delta_{n,m},
    \)
    where \(K_n \ne 0\) and \( \delta_{n,m} \) equals $1$  if \( n=m \), and \( 0 \) otherwise.
\end{itemize}
\end{defn}

Note that if \( \{P_n(x)\}_{n\ge0} \) is an OPS with respect to a
linear functional \( \LL \), then any sequence
\( \{ c_n P_n(x)\}_{n\ge0} \) with \( c_n\ne 0 \) is also an OPS with
respect to \( \LL \). Hence, for any OPS with respect to \( \LL \),
there is a monic OPS with respect to \( \LL \). Moreover, since
\( \{P_n(x)\}_{n\ge0} \) is also an OPS with respect to \( c\LL \) for
any \( c\ne0 \), we may assume that \( \LL(1)=1 \).

It is well known from Favard's theorem that a sequence of monic
polynomials \(\{P_n(x)\}_{n \ge 0}\) forms an OPS if and only if it
satisfies the three-term recurrence
\begin{equation}\label{eq:rec}
P_{n+1}(x) = (x - b_n)P_n(x) - \lambda_n P_{n-1}(x), \quad n \ge 0,
\end{equation}
with initial conditions \(P_{-1}(x)=0\) and \(P_0(x)=1\), where
\(\{b_n\}_{n \ge 0}\) and \(\{\lambda_n\}_{n \ge 1}\) are sequences
with \(\lambda_n \neq 0\). 

\begin{defn}\label{def:mixed_moment}
  Let \( \{P_n(x)\}_{n\ge0} \) be an OPS (not necessarily monic) with
  respect to \( \LL \). The \emph{moments} \(\sigma_{n}\) and the
  \emph{mixed moments} \(\sigma_{n,k}\) of this OPS are defined by
\[
 \sigma_{n} = \frac{\LL(x^n)}{\LL(P_0(x)^2)}, \qquad 
 \sigma_{n,k}=\frac{\LL\bigl(x^n P_k(x)\bigr)}{\LL\bigl(P_k(x)^2\bigr)}.
\]
\end{defn}

For consistency with mixed moments, we include the normalizing factor
in the definition of the moment \( \sigma_n \). Note that if
\( P_0(x)=1 \) and \( \LL(1) =1 \), then \( \sigma_n = \LL(x^n) \).

By orthogonality, we have
\begin{equation}\label{eq:5}
  x^n = \sum_{k=0}^{n} \sigma_{n,k} P_k(x).
\end{equation}
Since the linear functional \( \LL \) is uniquely determined by the
mixed moments \( \sigma_{n,k} \) up to a scalar multiple, \eqref{eq:5}
allows us to consider the orthogonal polynomials \( P_n(x) \) without
referring to \( \LL \). Viennot \cite{ViennotOP} found a combinatorial
interpretation for \(\sigma_{n,k}\) in terms of Motzkin paths.

The \emph{coefficients} \(\nu_{n,k}\) of \(P_n(x)\) are defined by
\[
P_n(x) = \sum_{k=0}^{n} \nu_{n,k} x^k.
\]
Since \(\{P_n(x)\}_{n \ge 0}\) and \(\{x^n\}_{n \ge 0}\) are bases of
the space of polynomials in \( x \), we have the following duality
between mixed moments and coefficients:
\begin{equation}\label{eq:duality-mixed-moments-coeff}
  \sum_{k \ge 0} \nu_{n,k} \sigma_{k,m} = \delta_{n,m},
  \qquad \sum_{k \ge 0} \sigma_{n,k} \nu_{k,m} = \delta_{n,m}.
\end{equation}

\begin{defn}\label{def:1}
  Let \( \{P_n(x)\}_{n \ge 0} \) be the OPS given by \eqref{eq:rec}.
  We say that \( \{P_n(x)\}_{n \ge 0} \) is the \emph{OPS with
    recurrence coefficients} \( (\{b_m\},\{\lambda_m\}) \), and write
\[
  P_n(x) = P_n(x;\{b_m\},\{\lambda_m\}).
\]
The moments \( \sigma_n \),
mixed moments \( \sigma_{n,k} \),
and coefficients \( \nu_{n,k} \)
of this OPS are written as
\[
  \sigma_n = \sigma_n(\{b_m\},\{\lambda_m\}), \qquad 
  \sigma_{n,k} = \sigma_{n,k}(\{b_m\},\{\lambda_m\}), \qquad 
  \nu_{n,k} = \nu_{n,k}(\{b_m\},\{\lambda_m\}).
\]
\end{defn}

One can easily convert a non-monic OPS to a monic OPS using a standard
technique; see \cite[p.~19]{Chihara} for example. For the reader's
convenience, we provide a precise statement with a proof.

\begin{lem}\label{lem:1}
  Suppose that \(\{p_n(x)\}_{n \ge 0}\) is an OPS satisfying the
  three-term recurrence
\begin{equation}\label{eq:1}
A_n p_{n+1}(x) + B_n p_n(x) + C_n p_{n-1}(x) = t x p_n(x), \qquad n \ge 0,
\end{equation}
with \( p_{-1}(x)=0 \) and \( p_{0}(x)=1 \), where \(t\) is a nonzero
constant and \( \{A_n\} \), \( \{B_n\} \), and \( \{C_n\} \) are
sequences with \(A_nC_n \ne 0\). Then the corresponding monic OPS
\(\{P_n(x)\}_{n \ge 0}\) is given by
\begin{equation}\label{eq:6}
  P_n(x) = \left( \prod_{k=0}^{n-1} A_k \right) t^{-n} p_n(x)
  = P_n(x;\{B_m/t\}, \{A_{m-1}C_m/t^2\}).
\end{equation}
Moreover, both OPSs \( \{p_n(x)\} \) and \( \{P_n(x)\} \) have the
same moments.
\end{lem}

\begin{proof}
  Let \( P_n(x) = ( \prod_{k=0}^{n-1} A_k ) t^{-n} p_n(x) \).
Multiplying \eqref{eq:1} by \( A_0 \cdots A_{n-1} t^{-n-1} \) yields
\[
  P_{n+1}(x) + \frac{B_n}{t} P_n(x) + \frac{A_{n-1}C_n}{t^2} P_{n-1}(x) = x P_n(x).
\]
Thus \( P_n(x) = P_n(x;\{B_m/t\}, \{A_{m-1}C_m/t^2\}) \), and the
polynomials \( P_n(x) \) are monic, which proves \eqref{eq:6}. The
second statement follows from the fact that \( \{p_n(x)\} \) and
\( \{P_n(x)\} \) are both OPSs with respect to the same linear
functional \( \LL \) and that \( p_0(x) = P_0(x) = 1 \).
\end{proof}

We note that by \Cref{lem:1}, an OPS \( \{p_n(x)\} \) with
\( p_0(x) = 1 \) and the corresponding monic OPS have the same
moments, but their mixed moments may be different.

In later sections, we will use the following lemmas, which show the
connection between the mixed moments of an OPS and its rescaled
version.

\begin{lem}\label{lem:2}
  Let \(\{P_n(x)\}_{n \ge 0}\) be any OPS with respect to a linear
  functional \( \LL \). For a nonzero constant \(t\), define
  \(Q_n(x) = t^{-n} P_n(tx)\) and \( \LL'(f(x)) = \LL(f(x/t)) \) for any
  polynomial \( f(x) \). Then \( \{Q_n(x)\}_{n \ge 0} \) is an OPS
  with respect to \( \LL' \). Moreover, if \(\sigma_{n,k}\) and
  \(\sigma'_{n,k}\) are the mixed moments of \(\{P_n(x)\}_{n \ge 0}\)
  and \(\{Q_n(x)\}_{n\ge 0}\), respectively, then
\[
\sigma'_{n,k} = t^{k-n}\sigma_{n,k}.
\]
In particular, we have \(\sigma'_n =\sigma'_{n,0} = t^{-n}\sigma_n\).
\end{lem}

\begin{proof}
  Since \( \LL'(Q_n(x) Q_m(x)) = t^{-n-m} \LL(P_n(x) P_m(x)) \), we
  obtain the first statement. Substituting \(x \mapsto tx\) in
  \eqref{eq:5} gives
\[
  (tx)^n = \sum_{k=0}^{n}\sigma_{n,k}P_k(tx) = \sum_{k=0}^{n}\sigma_{n,k} t^k Q_k(x).
\]
Dividing by \(t^n\) yields \(x^n = \sum_{k=0}^{n}t^{k-n}\sigma_{n,k}Q_k(x)\), hence
\(\sigma'_{n,k} = t^{k-n}\sigma_{n,k}\). 
\end{proof}

The monic version of \Cref{lem:2} can be restated using recurrence coefficients.

\begin{lem}\label{lem:5}
  For a nonzero constant \( t \) and recurrence coefficients
  \( (\{b_m\},\{\lambda_m\}) \), we have
  \[
    \sigma_{n,k}(\{t b_m\}, \{t^2 \lambda_m\})
    = t^{n-k} \sigma_{n,k}(\{b_m\}, \{\lambda_m\}).
  \]
  In particular, we have \( \sigma_{n}(\{t b_m\}, \{t^2 \lambda_m\}) =
t^{n} \sigma_{n}(\{b_m\}, \{\lambda_m\}) \).
\end{lem}

\begin{proof}
  Let \( \{P_n(x)\} \) and \( \{Q_n(x)\} \) be the OPSs given by
  \begin{align*}
    P_{n+1}(x) &= (x - tb_n)P_n(x) - t^2\lambda_n P_{n-1}(x),\\
    Q_{n+1}(x) &= (x - b_n)Q_n(x) - \lambda_n Q_{n-1}(x).
  \end{align*}
  Then \( Q_n(x) = t^{-n} P_n(tx) \), hence the result follows from \Cref{lem:2}.
\end{proof}

\subsection{Multivariate orthogonal polynomials}

In this subsection, we define multivariate orthogonal polynomials
using univariate orthogonal polynomials and relate the mixed moments
and coefficients of the univariate orthogonal polynomials to those of
the corresponding multivariate ones.

Let \(\x_n = (x_1, \dots, x_n)\) be a sequence of variables. For
\( \lambda\in \Par_n \), the \emph{Schur polynomial}
\(s_\lambda(\x_n)\) is defined via the bialternant formula:
\[
s_\lambda(\x_n) := \frac{\det \big( x_i^{\lambda_j + n - j} \big)_{i,j=1}^n}{\Delta(\x_n)},
\]
where \(\Delta(\x_n) = \prod_{1 \le i < j \le n} (x_i - x_j)\) is the
Vandermonde determinant. Analogously to this bialternant construction,
a univariate orthogonal polynomial sequence \( \{P_m(x)\}_{m\ge0} \)
can be lifted to a family of multivariate orthogonal polynomials
\(\{P_\lambda(\x_n)\}_{\lambda \in \Par_n}\) as follows.

\begin{defn}
  Let \( \{P_m(x)\}_{m\ge0} \) be orthogonal polynomials. The
  corresponding \emph{multivariate orthogonal polynomials}
  \(\{P_\lambda(\x_n)\}_{\lambda \in \Par_n}\) are defined by
\[
  P_\lambda(\x_n) := \frac{\det \big( P_{\lambda_j + n - j}(x_i)
    \big)_{i,j=1}^n}{\Delta(\x_n)}.
\]
The \emph{(multivariate) mixed moments} \(M_{\lambda,\mu}\) and the
\emph{(multivariate) coefficients} \(N_{\lambda,\mu}\) of the multivariate orthogonal
polynomials \(P_\lambda(\x_n)\) are defined by
\begin{align}
  \label{eq:14}
  s_\lambda(\x_n) &= \sum_{\mu\in \Par_n} M_{\lambda,\mu} \, P_\mu(\x_n),\\
  \label{eq:15}
  P_\lambda(\x_n) &= \sum_{\mu\in \Par_n} N_{\lambda,\mu} \, s_\mu(\x_n).
\end{align}
We also write \( M_{\lambda} := M_{\lambda,\emptyset} \)
and \( N_{\lambda} := N_{\lambda,\emptyset} \),
and call \( M_\lambda \) the \emph{(multivariate) moments} of the multivariate orthogonal polynomials.
\end{defn}

We caution the reader that it is important to specify \(n\) in the
definitions of \(M_{\lambda,\mu}\) and \(N_{\lambda,\mu}\) for
\(\lambda,\mu \in \Par_n\), because their values depend on the choice
of \(n\). For example, by \Cref{lem:4} below, we have
\[
  M_{(1),(0)}=\det\begin{pmatrix}
      \sigma_{1,0}
\end{pmatrix} = \sigma_{1,0}, \qquad
M_{(1,0),(0,0)}=\det\begin{pmatrix}
      \sigma_{2,1} &\sigma_{2,0}\\
      \sigma_{0,1} &\sigma_{0,0}
  \end{pmatrix}=\sigma_{2,1},
\]
which are not equal in general.

It is well known \cite[Section~5.1.1]{Dunkl2014} that if
\( \{P_m(x)\}_{m\ge0} \) is an OPS with respect to a linear functional
\( \LL_x \) on the space of polynomials in \( x \), then the
multivariate polynomials \( P_\lambda(\x_n) \) are orthogonal with
respect to the linear functional \( \LL_{\x_n} \) on the space of
polynomials in \( x_1,\dots,x_n \) defined by
\[
  \LL_{\x_n}(f(x_1,\dots,x_n))
  = (\LL_{x_n}\circ \cdots \circ \LL_{x_1})(f(x_1,\dots,x_n)).
\]
By the orthogonality, applying \( \LL_{\x_n} \) to \eqref{eq:14} gives
\[
  M_{\lambda} = M_{\lambda,\emptyset} = \LL_{\x_n} (s_\lambda(\x_n))/\LL_{\x_n} (P_\emptyset(\x_n)).
\]
Hence, our definition of moments of multivariate orthogonal
polynomials is consistent with that in \cite[Definition~2.6]{Corteel2019}.

Since both \( \{s_\lambda(\x_n)\}_{\lambda\in \Par_n} \) and
\( \{P_\lambda(\x_n)\}_{\lambda\in \Par_n} \) are bases of the space
of symmetric polynomials in \( \x_n \), we obtain the following proposition.

\begin{prop}\label{prop:transition}
  Let \(M_{\lambda,\mu}\) and \(N_{\lambda,\mu}\) be the mixed moments
  and coefficients of the multivariate orthogonal polynomials
  \(\{P_\lambda(\x_n)\}_{\lambda \in \Par_n}\), respectively. Then the
  matrices \(M = (M_{\lambda,\mu})_{\lambda,\mu\in \Par_n}\) and
  \(N = (N_{\lambda,\mu})_{\lambda,\mu\in \Par_n}\) are mutually
  inverse:
\[
\sum_{\rho\in \Par_n} M_{\lambda,\rho} N_{\rho,\mu} = \delta_{\lambda,\mu}, \qquad
\sum_{\rho\in \Par_n} N_{\lambda,\rho} M_{\rho,\mu} = \delta_{\lambda,\mu}.
\]
\end{prop}

The following proposition shows that the mixed moments and
coefficients of multivariate orthogonal polynomials can be expressed
as determinants of those of univariate orthogonal polynomials.

\begin{prop} \cite[Proposition 2.10]{LHT}
  \label{lem:4}
  Let \( \{P_m(x)\}_{m\ge0} \) be monic orthogonal polynomials with
  corresponding multivariate orthogonal polynomials
  \(\{P_\lambda(\x_n)\}_{\lambda \in \Par_n}\). Let \(\sigma_{n,k}\)
  and \(\nu_{n,k}\) be the mixed moments and coefficients of
  \( \{P_m(x)\}_{m\ge0} \), respectively, and let \(M_{\lambda,\mu}\)
  and \(N_{\lambda,\mu}\) be the mixed moments and coefficients of
  \(\{P_\lambda(\x_n)\}_{\lambda \in \Par_n}\), respectively. Then we
  have
  \[
    M_{\lambda,\mu} = \det\bigl(\sigma_{\lambda_i+n-i,\mu_j+n-j}\bigr)_{i,j=1}^n,
  \]
  \[
    N_{\lambda,\mu} = \det\bigl(\nu_{\lambda_i+n-i,\mu_j+n-j}\bigr)_{i,j=1}^n.
  \]
\end{prop}

Note that every \( n\times n \) minor of the matrix
\( (\sigma_{i,j})_{i,j\ge0} \) can be expressed as
\( M_{\lambda,\mu} \) for some \( \lambda,\mu\in \Par_n \). Since
\( (\sigma_{i,j})_{i,j\ge0} \) is lower triangular, we have
\( M_{\lambda,\mu}=0 \) unless \( \mu\subseteq\lambda \). Hence, it
suffices to consider \( M_{\lambda,\mu} \) and, similarly,
\( N_{\lambda,\mu} \) for \( \mu\subseteq\lambda \).

As a consequence of \Cref{lem:4}, the univariate mixed moments and coefficients arise as special cases of the multivariate ones, yielding the following corollary; see also \cite[Theorem~4.6]{Corteel2019}.

\begin{cor}\label{cor:5}
  Using the notation in \Cref{lem:4}, for \( n\geq k \) we have
\[
   \sigma_{n,k} = M_{(n-k,0^k),\emptyset},\quad 
   \nu_{n,k} = N_{(n-k,0^k),\emptyset}.
\]
\end{cor}

\begin{proof}
  Let \( \lambda=(n-k,0^k)\in \Par_{k+1} \) and
  \( \mu=\emptyset=(0^{k+1})\in \Par_{k+1} \). Since
\[
    M_{\lambda,\mu} = \det(\sigma_{\lambda_i+k+1-i,\mu_j+k+1-j})_{i,j=1}^{k+1}
  \]
  is the determinant of an upper triangular matrix whose diagonal
  entries are \( \sigma_{n,k}, 1,1,\dots,1 \), we have
  \( M_{\lambda,\mu} = \sigma_{n,k} \). The identity for \(\nu_{n,k}\)
  can be proved similarly.
\end{proof}

When \(\mu=\emptyset\), the multivariate mixed moment
\(M_{\lambda,\mu}\) admits a determinantal representation in terms of
the univariate moments \(\sigma_n\), as stated in the following
proposition.

\begin{prop}[{\cite[Theorem~4.6]{Corteel2019}}]\label{lem:6}
  We use the notation in \Cref{lem:4}. For 
  \(\lambda\in \Par_n\), we have
  \[
   M_{\lambda} =M_{\lambda,\emptyset} =  \frac{\det(\sigma_{\lambda_i+n-i+n-j})_{i,j=1}^n}{\det(\sigma_{n-i+n-j})_{i,j=1}^n}.
  \]
\end{prop}

\begin{remark}\label{rem:1}
\Cref{lem:6} shows that the quantity \( \mathcal{K}_\lambda \)
in \cite[(12)]{Corteel2019} is exactly the moment
\( M_{\lambda} \).
Moreover, in 
\cite[Theorem~5.1]{Corteel2019}, it is shown that
for \( \lambda=(\lambda_1,\dots,\lambda_n) \), we have
\[
\mathcal{K}_\lambda=\det\left(\mathcal{K}_{\left(\lambda_i+j-i, 0^{n-j}\right)}\right)_{i, j=1}^n,
\]
where we use the convention
\( \mathcal{K}_{(r,0^{n-j})} = M_{(r,0^{n-j}),\emptyset} := 0 \) for \( r<0 \).
This is immediate from \Cref{cor:5}, noting that if \( \lambda_i+j-i<0 \),
then \( \sigma_{\lambda_i+n-i,n-j}=0 \), which is consistent with this convention:
\[
\mathcal{K}_\lambda = M_{\lambda}
 = \det(\sigma_{\lambda_i+n-i,n-j})_{i,j=1}^n
 = \det(M_{(\lambda_i+j-i,0^{n-j}),\emptyset})_{i,j=1}^n
 = \det(\mathcal{K}_{(\lambda_i+j-i,0^{n-j})})_{i,j=1}^n.
\]
\end{remark}

For a matrix \(A\) and two finite subsets \(I\) and \(J\) of its row
and column indices with \(|I|=|J|\), let \([A]_{I,J}\) denote the
minor of \(A\) corresponding to the row set \(I\) and the column set
\(J\).

\begin{lem}\label{lem:inv_minor}
  Let \(A\) be an invertible \(n \times n\) matrix, and let
  \(I,J\subseteq\{0,\dots,n-1\}\) with \(|I|=|J|\). We have
  \[
    [A^{-1}]_{I,J}=(-1)^{\norm{I}+\norm{J}}\frac{[A]_{J^c,I^c}}{\det(A)},
  \]
  where \(\norm{I}=\sum_{i\in I}i\) and
  \(I^c=\{0,\dots,n-1\}\setminus I\).
\end{lem}

By \Cref{lem:inv_minor}, we can prove that every multivariate mixed
moment \( M_{\lambda,\mu} \) can also be written, up to sign, as the multivariate
coefficient \( N_{\tilde\mu,\tilde\lambda} \) indexed by some
partitions \( \tilde\mu \) and \( \tilde\lambda \), which is given in
the following lemma.

\begin{lem}\label{lem:10}
  Let \( \{P_r(x)\}_{r\ge0} \) be a monic OPS. Let
  \( \lambda, \mu  \in \Par_n \) with
  \( \lambda_1, \mu_1\le m \).
  Then
  \[
    M_{\lambda, \mu}
    =
    (-1)^{|\lambda| - |\mu|}
    N_{\widetilde\mu, \widetilde\lambda},
  \]
  where \( \widetilde\lambda, \widetilde\mu \in \Par_m \) are given by
  \( \widetilde\lambda_j = n - \lambda'_{m+1-j} \) and
  \( \widetilde\mu_j = n - \mu'_{m+1-j} \).
\end{lem}

\begin{proof}
  Let \( \sigma_{s,t} \) and \( \nu_{s,t} \) be the mixed moments and
  the coefficients of the OPS \( \{P_r(x)\}_{r\ge0} \), respectively.
  Let \( M = (\sigma_{s, t})_{s,t=0}^{m+n-1} \) and
  \( N = (\nu_{s, t})_{s,t=0}^{m+n-1} \). By \Cref{lem:4}, we
  have \(M_{\lambda,\mu}=[M]_{I,J}\), where
  \( I = \{ \lambda_i + n - i : 1\le i\le n\} \) and
  \( J = \{ \mu_i + n - i : 1\le i\le n\} \). By the
  duality~\eqref{eq:duality-mixed-moments-coeff}, \(M\) and \(N\) are
  mutually inverse. Since the OPS is monic, \(M\) is lower
  unitriangular, and hence has determinant \(1\). Therefore, by
  \Cref{lem:inv_minor}, we have
  \[
    M_{\lambda, \mu} = [M]_{I,J}
    = (-1)^{|\lambda|-|\mu|} [N]_{J^c, I^c},
  \]
  where \( I^c=\{ 0,\dots,m+n-1\} \setminus I \) and
  \( J^c=\{ 0,\dots,m+n-1\} \setminus J \).

  Since \(\lambda_1\le m \) and \( \lambda'_1\le n \), by the
  well-known fact \cite[Ch.I~(1.7)]{Macdonald}, we have
  \begin{align*}
    \{ \lambda_i+n-i : 1 \le i \le n \} \cup \{n-1+j-\lambda_j^{\prime} : 1 \le j \le m\} &= \{0,\ldots, m+n-1\}, \\
    \{ \lambda_i+n-i : 1 \le i \le n \} \cap \{n-1+j-\lambda_j^{\prime} : 1 \le j \le m\} &= \emptyset.
  \end{align*}
  Thus, \( I^c=\{n-1+j-\lambda'_j : 1 \le j \le m\} \).
  Similarly, \( J^c=\{n-1+j-\mu'_j : 1 \le j \le m\} \).

  On the other hand, \( N_{\widetilde\mu,\widetilde\lambda} = [N]_{\widetilde{J}, \widetilde{I}} \),
  where \( \widetilde{I} = \{\widetilde\lambda_j+m-j : 1 \le j \le m\} \) and
  \( \widetilde{J} = \{\widetilde\mu_j+m-j : 1 \le j \le m\} \).
  Then the \( j \)-th largest element of \( \widetilde{I} \)
  is 
  \[
    \widetilde\lambda_j+m-j = n + m - j - \lambda'_{m+1-j},
  \]
  which is the \( j \)-th largest element of \( I^c \).
  Hence, we have \( I^c = \widetilde{I} \), and 
  similarly we have \( J^c = \widetilde{J} \).
  Therefore, we obtain
  \( M_{\lambda, \mu} = (-1)^{|\lambda| - |\mu|}
  N_{\widetilde\mu, \widetilde\lambda} \).
\end{proof}

\subsection{2-ASEP}

The two-species asymmetric simple exclusion process (2-ASEP) is a
Markov chain describing the evolution of two types of particles (heavy
and light) on a one-dimensional lattice of \(n\) sites. Each site is
empty, occupied by a heavy particle, or occupied by a light
particle. Particles can hop left and right, and only
heavy particles can enter and exit the system. In the absence of light
particles, the model reduces to the classical ASEP. See
\cite{LHT,CMW,Corteel2020a} for further details.

In the long-time limit, the system converges to a stationary
distribution. Uchiyama \cite{Uchiyama2008} showed that this stationary
distribution can be computed using a matrix ansatz.
Suppose there
exist matrices \(D, E, A\), a row vector \(\langle W|\), and a column
vector \(|V\rangle\) satisfying the following relations:
\begin{itemize}
\item \(\langle W|(\alpha E-\gamma D)=\langle W|\),
\item \((\beta D-\delta E)|V\rangle=|V\rangle\),
\item \(D E-q E D=D+E\),
\item \(D A=q A D+A\),
\item \(A E=q E A+A\).
\end{itemize}

For a statement \( P \), define \( \vec1_{P} \) to be \( 1 \) if
\( P \) is true and \( 0 \) otherwise.
Write a state as a sequence \( (\tau_1,\dots,\tau_n) \) of \( 2 \)'s,
\( 1 \)'s, and \( 0 \)'s, which represent heavy particles, light
particles, and empty sites, respectively.
Uchiyama \cite{Uchiyama2008} showed that
the stationary probability of the state
\( (\tau_1,\dots,\tau_n) \)
is proportional to
\[
  \left\langle W\left|
      \prod_{i=1}^{n}(\vec1_{\tau_i=2}D + \vec1_{\tau_i=1}A + \vec1_{\tau_i=0}E)
    \right| V\right\rangle.
\]

Introducing the fugacity parameter \( \xi \), we define the \emph{fugacity
  partition function} \(Z_n\) and the \emph{refined fugacity partition
  function} \(Z_{n,k}\) as follows:
\begin{align}\label{eq:def-Z}
Z_n &= Z_n(\xi;\alpha,\beta,\gamma,\delta;q)=\frac{\left\langle W\left|(\xi D+E)^n\right| V\right\rangle}{\left\langle W | V\right\rangle},\\
  \label{eq:2}
Z_{n,k} &= 
Z_{n,k}(\xi;\alpha,\beta,\gamma,\delta;q)
= \left[t^k\right] \frac{\left\langle W\left|(\xi D+E+tA)^n\right| V\right\rangle}{\left\langle W\left|A^k\right| V\right\rangle},
\end{align}
where \( [t^k]F \) denotes the coefficient of \( t^k \) in \( F \).
Note that $Z_{n,0} = Z_n$.

Corteel, Mandelshtam, and Williams \cite{CMW} introduced rhombic
staircase tableaux, a combinatorial model that generalizes staircase
tableaux and rhombic alternative tableaux.
For \(n,k \ge 0\), let \(\RST_{n,k}\) denote the set of rhombic
staircase tableaux of size \((n,k)\); see \cite[Section~3]{CMW} for
the definition. Following \cite[Definition~3.14]{CMW}, we define
\begin{equation}\label{eq:def-tZ}
\widetilde{Z}_{n,k} = \widetilde{Z}_{n,k}(\xi;\alpha,\beta,\gamma,\delta;q) = \sum_{T \in \RST_{n,k}} \wt(T),
\end{equation}
and \(\widetilde{Z}_n = \widetilde{Z}_{n,0}\). If \( k=0 \),
then \( \RST_{n,0} \) reduces to the set of staircase tableaux of size
\(n\).
Corteel et al. \cite{CMW} derived the following
combinatorial formula for the partition function of the 2-ASEP.

\begin{lem} \cite[Lemma~8.4]{CMW}
  \label{lem:3}
  We have
\[
  Z_{n, k}(\xi;\alpha,\beta,\gamma,\delta;q)
  =\widetilde{Z}_{n, k}(\xi;\alpha,\beta,\gamma,\delta;q) \prod_{i=2k}^{n+k-1}\left(\alpha \beta-q^{i} \gamma \delta\right)^{-1}.
\]
\end{lem}

\section{Askey--Wilson moments and 2-ASEP}
\label{sec:askey-wilson-moments}

This section collects several known results scattered in the
literature and presents them in a unified notation. Our goal is to
clarify the relation between Askey--Wilson polynomials and the
partition function of the two-species ASEP, a result established in \cite{Corteel2012, Corteel2019}. More precisely, in
\Cref{subsec:AW2ASEP}, we introduce rescaled Askey--Wilson polynomials
and identify \(Z_n\) and \(Z_{n,k}\) with the corresponding moment and
mixed moment, respectively. In \Cref{subsec:explicit_alpha}, we
rewrite the recurrence coefficients of the rescaled Askey--Wilson
polynomials explicitly in terms of the ASEP parameters
\(\alpha,\beta,\gamma \), and \(\delta\) and provide the specializations needed
later.

\subsection{Askey--Wilson polynomials}
\label{subsec:AW2ASEP}

We begin with the definition of the monic Askey--Wilson polynomials.

\begin{defn}\label{def:monic-AW}
  The monic \emph{Askey--Wilson polynomials}
  \(P^{AW}_n(x)=P^{AW}_n(x;a,b,c,d|q)\) are defined by
\[
P_{n+1}^{AW}(x)=\left(x-\frac{B_n}{2}\right)P_n^{AW}(x)-\frac{A_{n-1}C_n}{4}P_{n-1}^{AW}(x),
\]
with \(P^{AW}_0(x)=1\) and \(P^{AW}_{-1}(x)=0\), where
\[
\begin{aligned}
A_n&=\frac{1-q^{n-1}abcd}{(1-q^{2n-1}abcd)(1-q^{2n}abcd)}, \\
B_n&=\frac{q^{n-1}\left((1+q^{2n-1}abcd)(qs+abcd s')-q^{n-1}(1+q)abcd(s+qs')\right)}{(1-q^{2n-2}abcd)(1-q^{2n}abcd)}, \\
C_n&=\frac{(1-q^n)(1-q^{n-1}ab)(1-q^{n-1}ac)(1-q^{n-1}ad)(1-q^{n-1}bc)(1-q^{n-1}bd)(1-q^{n-1}cd)}{(1-q^{2n-2}abcd)(1-q^{2n-1}abcd)},
\end{aligned}
\]
and
\[
s=a+b+c+d, \qquad s'=a^{-1}+b^{-1}+c^{-1}+d^{-1}.
\]
We denote by
\( \sigma_n^{AW} \),
\( \sigma_{n,k}^{AW} \), and
\( \nu_{n,k}^{AW} \)
the moments, mixed moments,
and coefficients of \( P^{AW}_n(x) \), respectively:
\begin{align*}
 \sigma_n^{AW} &= \sigma_n^{AW}(a,b,c,d;q) = \sigma_n(\{B_m/2\},\{A_{m-1}C_m/4\}), \\
 \sigma_{n,k}^{AW} &= \sigma_{n,k}^{AW}(a,b,c,d;q) = \sigma_{n,k}(\{B_m/2\},\{A_{m-1}C_m/4\}),\\
 \nu_{n,k}^{AW} &= \nu_{n,k}^{AW}(a,b,c,d;q) = \nu_{n,k}(\{B_m/2\},\{A_{m-1}C_m/4\}).
\end{align*}
\end{defn}

\begin{remark}\label{rem:5}
  In \cite[(2.5)]{Corteel2012}, non-monic Askey--Wilson polynomials
  are used. Using \Cref{lem:1}, one can convert them to the monic
  version in \Cref{def:monic-AW}.
\end{remark}

We recall the reparametrization between the Askey--Wilson parameters
\((a,b,c,d)\) and the ASEP parameters \((\alpha,\beta,\gamma,\delta)\),
which will be used repeatedly throughout this paper.

\begin{defn}\label{def:alpha} \cite[p.395]{Corteel2011}
Let
\[
  \alpha  =\frac{1-q}{1+ac+a+c}, \quad \beta=\frac{1-q}{1+bd+b+d},
  \quad \gamma=\frac{-(1-q) a c}{1+ac+a+c}, \quad \delta  =\frac{-(1-q) b d}{1+bd+b+d}.
\]
Equivalently, 
\begin{align*}
  a&=\frac{1-q-\alpha+\gamma+\sqrt{(1-q-\alpha+\gamma)^2+4 \alpha \gamma}}{2 \alpha}, \\
  b&=\frac{1-q-\beta+\delta+\sqrt{(1-q-\beta+\delta)^2+4 \beta \delta}}{2 \beta}, \\
  c&=\frac{1-q-\alpha+\gamma-\sqrt{(1-q-\alpha+\gamma)^2+4 \alpha \gamma}}{2 \alpha}, \\
  d&=\frac{1-q-\beta+\delta-\sqrt{(1-q-\beta+\delta)^2+4 \beta \delta}}{2 \beta}.
\end{align*}
\end{defn}

We now define the following rescaled Askey--Wilson polynomials, whose
moment and mixed moment will be shown to be \(Z_n\) and
\(Z_{n,k}\), respectively. 

\begin{defn}\label{def:GZ}
  The \emph{rescaled Askey--Wilson polynomials}
  \(P^Z_n(x)=P^Z_n(x;\xi;\alpha,\beta,\gamma,\delta|q)\) are defined by
\[
P^Z_{n+1}(x)= \left(x-b_n^Z\right)P^Z_n(x)-\lambda_n^Z P^Z_{n-1}(x),\qquad
P^Z_0(x)=1,\quad P^Z_{-1}(x)=0,
\]
where
\begin{align}
  \label{eq:bZ}
b_n^{Z}&=b_n^{Z}(\xi;\alpha,\beta,\gamma,\delta;q)=\frac{\sqrt{\xi}B_n'+1+\xi}{1-q},\\
  \label{eq:laZ}
\lambda_n^{Z}&=\lambda_n^{Z}(\xi;\alpha,\beta,\gamma,\delta;q)=\frac{\xi A_{n-1}'C_n'}{(1-q)^2}.
\end{align}
Here \(A_n'\), \(B_n'\), and \(C_n'\) are obtained from \( A_n,B_n, \) and
\( C_n \) defined in \Cref{def:monic-AW} by replacing \(a\), \(b\),
\(c\), and \(d\) with \(a/\sqrt{\xi}, b\sqrt{\xi}, c/\sqrt{\xi}\), and
\(d\sqrt{\xi}\), respectively, and then changing the parameters
\( (a,b,c,d) \) to \( (\alpha,\beta,\gamma,\delta) \) using
\Cref{def:alpha}.

We denote by \( \sigma_n^Z \), \( \sigma_{n,k}^Z \), and
\( \nu_{n,k}^Z \) the moments, mixed moments, and coefficients of
\( P^Z_n(x) \), respectively:
\begin{align*}
 \sigma_n^Z &= \sigma_n^Z(\xi;\alpha,\beta,\gamma,\delta;q) = \sigma_n(\{b_m^Z\},\{\lambda_m^Z\}), \\
 \sigma_{n,k}^Z &= \sigma_{n,k}^Z(\xi;\alpha,\beta,\gamma,\delta;q) = \sigma_{n,k}(\{b_m^Z\},\{\lambda_m^Z\}),\\
 \nu_{n,k}^Z &= \nu_{n,k}^Z(\xi;\alpha,\beta,\gamma,\delta;q) = \nu_{n,k}(\{b_m^Z\},\{\lambda_m^Z\}).
\end{align*}
\end{defn}

Comparing the recurrence coefficients, we have
\begin{equation} \label{eq:monic_Z}
P^Z_n(x)
=
\left(\frac{2\sqrt \xi}{1-q}\right)^n
P_n^{AW}\left(\frac{(1-q)x-(1+\xi)}{2\sqrt \xi};
\frac{a}{\sqrt \xi}, b\sqrt \xi, \frac{c}{\sqrt \xi}, d\sqrt \xi \bigg| q\right).
\end{equation}

The following theorem rewrites \cite[Theorem~1.11]{Corteel2012} in our notation.
We present the proof to explain how their normalization translates to ours
and to identify the recurrence coefficients \(b_n^Z\) and \(\lambda_n^Z\).

\begin{thm}\label{thm:Z-sigma}
  The partition function \(Z_n(\xi;\alpha,\beta,\gamma,\delta;q)\) is
  the \(n\)th moment of the rescaled Askey--Wilson polynomials. That
  is,
\begin{equation*}
Z_n(\xi;\alpha,\beta,\gamma,\delta;q)
=
\sigma_n^Z(\xi;\alpha,\beta,\gamma,\delta;q).
\end{equation*}
\end{thm}

\begin{proof}
  First, we caution the reader that the notation \( Z_n \) used in
  \cite[Theorem~1.11]{Corteel2012} means \( \widetilde{Z}_n \) in our
  notation. Note that by \Cref{lem:3}, we have
\[
  \widetilde{Z}_n(\xi;\alpha,\beta,\gamma,\delta;q) = 
  Z_n(\xi;\alpha,\beta,\gamma,\delta;q) \prod_{j=0}^{n-1}
(\alpha\beta-\gamma\delta q^j).
\]
By \Cref{def:alpha}, the factor
\(\prod_{j=0}^{n-1} (\alpha\beta-\gamma\delta q^j)\) is equal to
\((abcd;q)_n(\alpha\beta)^n\). Therefore,
\cite[Theorem~1.11]{Corteel2012} can be rewritten as
\[
  Z_n(\xi;\alpha,\beta,\gamma,\delta;q)=\xi^{n/2} \sigma_{n}(\{b_m^Z/\sqrt{\xi} \}, \{\lambda_m^Z/\xi\}).
\]
Then, by \Cref{lem:5} with \(t=1/\sqrt{\xi}\), we obtain the result.
\end{proof}

We next show that \(Z_{n,k}\) is equal to the mixed moment of the
rescaled Askey--Wilson polynomials. This is essentially equivalent to
the result of Corteel and Williams \cite[Corollary~6.2]{Corteel2019}.
Since their formulation leaves a factor implicit, we state the identity using our
notation in a slightly simpler form with all factors explicit, and
include a proof.

\begin{thm}\label{thm:Z_nk}
  The refined fugacity partition function
  \(Z_{n,k}(\xi;\alpha,\beta,\gamma,\delta;q)\) is the mixed moment of
  the rescaled Askey--Wilson polynomials. That is,
\[
Z_{n,k}(\xi;\alpha,\beta,\gamma,\delta;q)
=
\sigma_{n,k}^Z(\xi;\alpha,\beta,\gamma,\delta;q).
\]
\end{thm}

\begin{proof}
  If \(n<k\), then both sides are equal to \( 0 \). Thus we may assume
  \(n\ge k\). Let \(\lambda=(n-k,0^k)\). By \cite[Theorem~1.1,
  Definition~4.2]{Corteel2019},
\[
Z_{n,k}(\xi;\alpha,\beta,\gamma,\delta;q)
=
\frac{
\det\left(
Z_{\lambda_i+2k+2-i-j}(\xi;\alpha,\beta,\gamma,\delta;q)
\right)_{i,j=1}^{k+1}
}{
\det\left(
Z_{2k+2-i-j}(\xi;\alpha,\beta,\gamma,\delta;q)
\right)_{i,j=1}^{k+1}
}.
\]
Applying~\Cref{thm:Z-sigma} to the right-hand side, we have
\[
Z_{n,k}(\xi;\alpha,\beta,\gamma,\delta;q)
=
\frac{
\det\left(
\sigma_{\lambda_i+2k+2-i-j}^Z
\right)_{i,j=1}^{k+1}
}{
\det\left(
\sigma_{2k+2-i-j}^Z
\right)_{i,j=1}^{k+1}
},
\]
which is equal to the corresponding multivariate mixed moment
\(M_{\lambda,\emptyset}\) by \Cref{lem:6}. Then \Cref{cor:5} yields
\(M_{\lambda,\emptyset}=\sigma_{n,k}^Z\).
This completes the proof.
\end{proof}

\subsection{Recurrence coefficients using ASEP parameters}
\label{subsec:explicit_alpha}

We rewrite the recurrence coefficients \(b_n^Z\)
and \(\lambda_n^Z\) of the rescaled Askey--Wilson polynomials
defined in \Cref{def:GZ} in terms of the ASEP parameters
\(\alpha,\beta,\gamma,\delta\).

\begin{prop}\label{prop:reparametrization}
 We have
\begin{align}
  \label{eq:bZ2}
  b_n^Z
&=
\frac{P_1(n)+P_2(n)}
{(\alpha\beta-\gamma\delta q^{2n-2})(\alpha\beta-\gamma\delta q^{2n})},\\
  \label{eq:laZ2}
\lambda_n^Z
&=
\frac{(1-q^{n})(\alpha\beta-\gamma\delta q^{n-2})(\xi\alpha+\gamma q^{n-1})(\beta+\xi\delta q^{n-1})(Q_1(n)+Q_2(n))}
{(1-q)(\alpha\beta-\gamma\delta q^{2n-2})^2(\alpha\beta-\gamma\delta q^{2n-3})(\alpha\beta-\gamma\delta q^{2n-1})},
\end{align}
where
\begin{align*}
P_1(n)
&=
(\xi\alpha+\beta)\left(\gamma\delta(q^{3n-1}-q^{2n-1}-q^{2n-2})+\alpha\beta q^n\right)\\
&\qquad
+(\gamma+\xi\delta)\left(-\gamma\delta q^{3n-2}+\alpha\beta(q^{2n}+q^{2n-1}-q^{n-1})\right),\\[2mm]
P_2(n)
&=
\frac{(1-q^n)(\alpha\beta-\gamma\delta q^{n-1})}{1-q}
\left((1+\xi)\gamma\delta q^{2n-1}+(\beta\gamma+\xi\alpha\delta)(q^n+q^{n-1})+(1+\xi)\alpha\beta\right),\\[2mm]
Q_1(n)
&=
\frac{(1-q^{n-1})(\alpha+\gamma q^{n-1})(\beta+\delta q^{n-1})(\alpha\beta-\gamma\delta q^{n-1})}{1-q},\\[2mm]
Q_2(n)
&=
q^{n-1}\Bigl(
(\alpha+\beta)(\gamma\delta q^{2n-2}-2\gamma\delta q^{n-1}+\alpha\beta)\\
&\qquad\qquad
-(\gamma+\delta)(\gamma\delta q^{2n-2}-2\alpha\beta q^{n-1}+\alpha\beta)
-(1-q)(\alpha+\delta q^{n-1})(\beta+\gamma q^{n-1})
\Bigr).
\end{align*}
\end{prop}
\begin{proof}
  Let
  \((a',b',c',d')=(a/\sqrt{\xi}, b\sqrt{\xi}, c/\sqrt{\xi}, d\sqrt{\xi})\). By \Cref{def:alpha}, we have
\[
a'b'c'd'
=
\frac{\gamma\delta}{\alpha\beta},
\qquad
a'c'=-\frac{\gamma}{\alpha \xi},
\qquad
b'd'=-\frac{\xi\delta}{\beta}.
\]
Thus, simplifying \eqref{eq:bZ} gives the first identity
\eqref{eq:bZ2}. For the second identity \eqref{eq:laZ2}, note that by \eqref{eq:laZ},
we have \(\lambda_n^Z={\xi A'_{n-1}C'_n}/{(1-q)^2}\), where
\[
A'_{n-1}=
\frac{\alpha\beta(\alpha\beta-\gamma\delta q^{n-2})}
{(\alpha\beta-\gamma\delta q^{2n-3})(\alpha\beta-\gamma\delta q^{2n-2})},
\qquad 
C'_n=
\frac{(1-q^n)\alpha\beta(\xi\alpha+\gamma q^{n-1})(\beta+\xi\delta q^{n-1})E_n}
{\xi(\alpha\beta-\gamma\delta q^{2n-2})(\alpha\beta-\gamma\delta q^{2n-1})},
\]
and
\( E_n= (1-q^{n-1}ab) (1-q^{n-1}ad) (1-q^{n-1}bc) (1-q^{n-1}cd) \).
Expanding \(E_n\) and using \(a+c=(1-q-\alpha+\gamma)/{\alpha}\),
\(b+d=(1-q-\beta+\delta)/{\beta}\), \(ac=-{\gamma}/{\alpha}\), and
\(bd=-{\delta}/{\beta}\), we obtain
\(\alpha^2\beta^2E_n=(1-q)(Q_1(n)+Q_2(n))\). Combining these gives
\eqref{eq:laZ2}.
\end{proof}

As corollaries of \Cref{prop:reparametrization}, we give explicit
formulas for \(b_n^Z\) and \(\lambda_n^Z\) in the special cases
\(q=1\) and \(q=0\), which will be used later in \Cref{sec:proof-special-cases}.

\begin{cor}\label{cor:q=1}
  If \( q=1 \), we have
\begin{align*}
b_n^Z
&=
\frac{
\xi\alpha+\beta+\gamma+\xi\delta
+
n((1+\xi)(\alpha\beta+\gamma\delta)+2(\beta\gamma+\xi\alpha\delta))
}{
\alpha\beta-\gamma\delta
},
\\
\lambda_n^Z
&=
\frac{
n(\xi\alpha+\gamma)(\beta+\xi\delta)
((n-1)(\alpha+\gamma)(\beta+\delta)+\alpha+\beta+\gamma+\delta)
}{
(\alpha\beta-\gamma\delta)^2
}.
\end{align*}
\end{cor}

\begin{cor}\label{cor:q=0}
   If \( q=0 \), we have
\begin{align*}
b_n^Z
&= \begin{cases}
\displaystyle
\frac{(\alpha+\delta)\xi+\beta+\gamma}{\alpha\beta-\gamma\delta},
& \mbox{if \( n=0 \)},\\[0.8em]
\displaystyle
\frac{-(\alpha \xi+\beta)\gamma\delta-(\gamma+\delta \xi)\alpha\beta+(\alpha\beta-\gamma\delta)(\beta\gamma+\alpha\delta \xi+(1+\xi)\alpha\beta)}
{(\alpha\beta-\gamma\delta)\alpha\beta},
& \mbox{if \( n=1 \)},\\[0.8em]
1+\xi,
& \mbox{if \( n\ge2 \),}
\end{cases}\\
\lambda_n^Z
&= \begin{cases}
\displaystyle
\frac{
(\xi\alpha+\gamma)(\beta+\xi\delta)
((\alpha+\beta+\gamma+\delta)(\alpha\beta-\gamma\delta)
-(\alpha+\delta)(\beta+\gamma))
}
{(\alpha\beta-\gamma\delta)^2\alpha\beta},
& \mbox{if \( n=1 \)},\\[0.8em]
\displaystyle
\frac{(\alpha\beta-\gamma\delta)\xi}{\alpha\beta},
& \mbox{if \( n=2 \)},\\[0.8em]
\xi,
& \mbox{if \( n\ge3 \).}
\end{cases}
\end{align*}
\end{cor}

Now we present an explicit formula for the coefficient \( \nu_{n,k}^Z \)
of the rescaled Askey--Wilson polynomials when $q=0$.
We first define the quantity
$V_{n,k} = V_{n,k}(\xi;\alpha,\beta,\gamma,\delta)$ by
\begin{align}
\label{eq:V} V_{n,k} &= \sum_{i=0}^{n} \binom{n-i-1}{k-1} \binom{k+i-1}{k-1}\xi^i  \alpha\beta
    + \sum_{i=0}^{n} \binom{n-i-1}{k-1} \binom{k+i-2}{k-1}\xi^i  \alpha\delta \\
  \notag
  &\quad+  \sum_{i=0}^{n-1} \binom{n-i-2}{k-1} \binom{k+i-1}{k-1}\xi^i  \beta\gamma
    + \sum_{i=0}^{n-1} \binom{n-i-2}{k-1} \binom{k+i-2}{k-1}\xi^i  \gamma\delta \\
  \notag
  &\quad+  \sum_{i=0}^{n} \binom{n-i-1}{k-1} \binom{k+i-1}{k}\xi^i  \alpha
    + \sum_{i=0}^{n-1} \binom{n-i-2}{k-1} \binom{k+i-1}{k}\xi^i  \gamma \\
  \notag
  &\quad+  \sum_{i=0}^{n} \binom{n-i-1}{k} \binom{k+i-1}{k-1}\xi^i  \beta
    + \sum_{i=1}^{n} \binom{n-i-1}{k} \binom{k+i-2}{k-1}\xi^i  \delta \\
  \notag
  &\quad+  \sum_{i=0}^{n} \binom{n-i-1}{k} \binom{k+i-1}{k}\xi^i.
\end{align}
Here and throughout the paper, we use the convention that $\binom{a}{b}=0$ whenever $a<0$ or $b<0$, with the single exception $\binom{-1}{-1}=1$.

\begin{prop}\label{prop:7}
  If \( q = 0 \), we have
\[
(-1)^{n-k}\cdot\nu_{n,k}^Z
= \begin{cases}
  \displaystyle
  \vspace{.2cm}
 \frac{(\alpha  + \delta) \xi + \beta + \gamma}{\alpha\beta - \gamma \delta} & \mbox{if \( (n,k)=(1,0) \)},\\
  \vspace{.2cm}
  \displaystyle
\frac{\alpha \xi^2+(\gamma+\delta+1)\xi+\beta}{\alpha\beta} & \mbox{if \( (n,k)=(2,0) \)},\\
  \displaystyle
\frac{V_{n,k}(\xi;\alpha,\beta,\gamma,\delta)}{\alpha\beta}& \mbox{otherwise.}
  \end{cases}
\]
\end{prop}

\begin{proof}
Let \(u_{n,k} = (-1)^{n-k}\alpha\beta\cdot \nu_{n,k}^Z\).
Then we need to show that
\begin{equation}\label{eq:30}
u_{n,k}
= \begin{cases}
  \displaystyle
  \vspace{.2cm}
 \frac{((\alpha  + \delta) \xi + \beta + \gamma)\alpha\beta}{\alpha\beta - \gamma \delta} & \mbox{if \( (n,k)=(1,0) \)},\\
  \vspace{.2cm}
  \displaystyle
\alpha \xi^2+(\gamma+\delta+1)\xi+\beta & \mbox{if \( (n,k)=(2,0) \)},\\
  \displaystyle
V_{n,k}(\xi;\alpha,\beta,\gamma,\delta)& \mbox{otherwise.}
  \end{cases}
\end{equation}
Since \( P_n^Z(x) = \sum_{k=0}^n \nu_{n,k}^Zx^k  \), by \Cref{def:GZ}, we have
\begin{equation}\label{eq:32}
  \nu_{n+1,k}^Z=\nu_{n,k-1}^Z-b_{n}^Z\nu_{n,k}^Z-\lambda_{n}^Z\nu_{n-1,k}^Z \qquad (n\ge0).
\end{equation}
Thus, using \Cref{cor:q=0}, we compute \( u_{n,k} \) for \( 0\le n\le 4 \) as follows:
\begin{align*}
&u_{n,n}=\alpha\beta \qquad (0\le n\le 4), \\
&u_{1,0}=\frac{((\alpha  + \delta) \xi + \beta + \gamma)\alpha\beta}{\alpha\beta - \gamma \delta},\\
&u_{2,0}=\alpha \xi^2+(\gamma+\delta+1)\xi+\beta,\\
&u_{2,1}=(\xi+1)\alpha\beta+\xi\alpha\delta+\beta\gamma+\xi\alpha+\beta,\\
&u_{3,0}=\alpha \xi^3 +(\gamma +1)\xi^2+(\delta+1) \xi + \beta,\\
&u_{3,1}= (\xi^2+\xi+1)\alpha\beta +(\xi^2+\xi)\alpha\delta+(\xi+1)\beta\gamma+\xi \gamma\delta  +(2\xi^2+\xi)\alpha+(\xi+2)\beta+\xi \gamma+\xi \delta +\xi,\\
&
u_{3,2}= 2(\xi+1)\alpha\beta+\xi\alpha\delta+\beta\gamma+\xi\alpha +\beta, \\
&u_{4,0}=\alpha \xi^4 +(\gamma+1)\xi^3+\xi^2+(\delta+1)\xi+\beta,\\
&u_{4,1}= (\xi^3+\xi^2+\xi+1)\alpha\beta +(\xi^3+\xi^2+\xi)\alpha\delta+(\xi^2+\xi+1)\beta\gamma+(\xi^2+\xi)\gamma\delta\\
&\qquad\quad +(3\xi^3+2\xi^2+\xi)\alpha+(\xi^2+2\xi+3)\beta+(2\xi^2+\xi)\gamma+(\xi^2+2\xi)\delta+2\xi^2+2\xi,\\
&u_{4,2}= (3\xi^2+4\xi+3)\alpha\beta+(2\xi^2+2\xi)\alpha\delta+2(\xi+1)\beta\gamma+\xi\gamma\delta+(3\xi^2+2\xi)\alpha+(2\xi+3)\beta\\
&\qquad\quad +\xi\gamma+\xi\delta+\xi,\\
&u_{4,3}= 3(\xi+1)\alpha\beta+\xi\alpha\delta+\beta\gamma+\xi\alpha+\beta.
\end{align*}
Hence, \eqref{eq:30} holds for \( n\le 4 \).

Now let \( n\ge4 \) and suppose that \eqref{eq:30} holds for all
integers at most \( n \). By \eqref{eq:32} and \Cref{cor:q=0}, we have
\begin{equation}\label{eq:31}
  u_{n+1,k}=u_{n,k-1}+(1+\xi) u_{n,k}-\xi u_{n-1,k} \qquad (n\ge 3).
\end{equation}
By the inductive hypothesis, we have
\( u_{r,s} = V_{r,s}(\xi;\alpha,\beta,\gamma,\delta) \) for
\( r\in \{n,n-1\} \), and therefore,
\begin{align*}
[\alpha\beta]u_{n+1,k}&=
[\alpha\beta]u_{n,k-1}+(1+\xi)[\alpha\beta]u_{n,k}-[\alpha\beta]\xi u_{n-1,k}\\
&=\sum_{i=0}^{n} \binom{n-i-1}{k-2} \binom{k+i-2}{k-2}\xi^i  
+\sum_{i=0}^{n} \binom{n-i-1}{k-1} \binom{k+i-1}{k-1}\xi^i  \\
  &\quad +\sum_{i=1}^{n+1} \binom{n-i}{k-1} \binom{k+i-2}{k-1}\xi^i
  -\sum_{i=1}^{n+1} \binom{n-i-1}{k-1} \binom{k+i-2}{k-1}\xi^i\\
&=\sum_{i=0}^{n} \binom{n-i-1}{k-2} \binom{k+i-2}{k-2}\xi^i  
+\sum_{i=0}^{n} \binom{n-i-1}{k-1} \binom{k+i-1}{k-1}\xi^i\\
  &\quad +\sum_{i=0}^{n} \binom{n-i-1}{k-2} \binom{k+i-2}{k-1}\xi^i\\
&=\sum_{i=0}^{n} \binom{n-i-1}{k-2} \binom{k+i-1}{k-1}\xi^i  
+\sum_{i=0}^{n} \binom{n-i-1}{k-1} \binom{k+i-1}{k-1}\xi^i\\
&=\sum_{i=0}^{n} \binom{n-i}{k-1} \binom{k+i-1}{k-1}\xi^i \\
&=[\alpha\beta]V_{n+1,k}(\xi;\alpha,\beta,\gamma,\delta).
\end{align*}
Here \( [s]f \) denotes the coefficient of the monomial \(s\) in \(f\), viewed as a polynomial in \(\alpha,\beta,\gamma,\delta\).
A similar computation applies to each
\(s\in\{\alpha\beta,\alpha\delta,\beta\gamma,\gamma\delta,\alpha,\beta,\gamma,\delta,1\}\).
Hence, \eqref{eq:30} also holds for \( n+1 \).
By induction, the proof is complete.
\end{proof}

We finish this section with the following corollary of \Cref{prop:7},
which will be useful in \Cref{subsec:q0-xi1}.

\begin{cor}\label{cor: coe xi=1 q=0}
  Suppose \( \xi=1 \) and \( q=0 \).
  Then 
  \[
    \nu_{1,0}^Z = -\frac{\alpha+\beta+\gamma+\delta}{\alpha\beta-\gamma\delta}, \qquad 
   \nu_{2,0}^Z = \frac{\alpha+\beta+\gamma+\delta+1}{\alpha\beta},
  \]
  and for \((n,k)\ne(1,0),(2,0)\), we have
\[
\begin{aligned}
\nu_{n,k}^Z
=
\frac{(-1)^{n-k}}{\alpha\beta}
\biggl(
&\binom{n+k-1}{2k-1}\alpha\beta
+\binom{n+k-2}{2k-1}(\alpha\delta+\beta\gamma)
+\binom{n+k-3}{2k-1}\gamma\delta \\
&+\binom{n+k-1}{2k}(\alpha+\beta)
+\binom{n+k-2}{2k}(\gamma+\delta)
+\binom{n+k-1}{2k+1}
\biggr).
\end{aligned}
\]
\end{cor}

\section{Koornwinder moments and Rains' conjecture}
\label{sec:koornw-moments}

In this section, we study the Koornwinder polynomials and the
multivariate Askey--Wilson polynomials. In
\Cref{sec:koornw-polyn-rains}, we introduce Koornwinder moments
and Rains' conjecture. We then present a generalization of Rains'
conjecture for certain parameter values. In \Cref{sec:deno}, we
establish explicit formulas for the multivariate moments
\(M^{AW}_\lambda\) and \(M^{Z}_\lambda\) coming from the monic and
rescaled Askey--Wilson polynomials \(P^{AW}_n(x)\) and \(P^{Z}_n(x)\),
respectively (\Cref{prop:explicit_formula} and
\Cref{cor:explicit_Koornwinder_moment}). As an application, we
determine the minimal denominators of these moments, which provides an
explicit formulation of Rains' conjecture.

\subsection{Koornwinder moments and Rains' conjecture}
\label{sec:koornw-polyn-rains}

Koornwinder polynomials are multivariate orthogonal Laurent
polynomials defined as follows.

\begin{defn} \cite[Definition~5.1]{Koornwinder1992}
For \( \lambda\in\Par_n \), the \emph{Koornwinder polynomial} \( P_\lambda(\vec z_n; a, b, c, d | q, t) \) is
  the unique Laurent polynomial that is invariant under permutation and inversion of variables,
  with leading coefficient \( z_1^{\lambda_1} \cdots z_n^{\lambda_n} \), and orthogonal with respect to the \emph{Koornwinder density}
  \[
    \prod_{1 \le i < j \le n} \frac{(z_iz_j, z_i/z_j, z_j/z_i, 1/z_iz_j; q)_{\infty}}
    {(tz_iz_j, tz_i/z_j, tz_j/z_i, t/z_iz_j; q)_{\infty}}
    \prod_{1 \le i \le n} \frac{(z_i^2, 1/z_i^2;q)_{\infty}}
    {(az_i, a/z_i, bz_i, b/z_i,cz_i,c/z_i,dz_i,d/z_i;q)_{\infty}}
  \]
  on the unit torus \( \left| z_1 \right| = \cdots = \left| z_n \right| = 1 \).
\end{defn}

Koornwinder polynomials are also known as Macdonald--Koornwinder polynomials of type \( BC \).
A key feature of Koornwinder polynomials is that they interpolate several well-known families of orthogonal polynomials.
In particular, under the specialization \( t = q \),
the Koornwinder polynomials reduce to multivariate Askey--Wilson polynomials
\cite[Definition~2.4]{Corteel2019}:
\[
  P_\lambda(\vec z_n; a, b, c, d | q, q) = 2^{|\lambda|} P_\lambda^{AW}(\x_n; a, b, c, d|q)
  = 2^{|\lambda|}\frac{\det ( P^{AW}_{\lambda_j + n - j}(x_i))_{i,j=1}^n}{\Delta(\x_n)},
\]
where \( \vec z_n=(z_1,\dots,z_n) \) and
\( \vec x_n=(x_1,\dots,x_n) \) are related by
\( x_i = (z_i+z_i^{-1})/2 \).

Motivated by this and the connection with ASEP, Corteel and Williams
\cite[Definition~4.2]{Corteel2019} define the Koornwinder moments at
\( t=q \) to be the moments of rescaled multivariate Askey--Wilson
polynomials. We give a precise definition as follows. Recall
\Cref{def:GZ} for the definition of the rescaled Askey--Wilson
polynomials \( P^Z_n(x) \).

\begin{defn}\label{def:2}
The \emph{rescaled multivariate Askey--Wilson polynomials} \( P^Z_{\lambda}(\x_n) \) are
defined by
\[
  P^Z_{\lambda}(\x_n) = P^Z_{\lambda}(\x_n;\xi;\alpha,\beta,\gamma,\delta|q)
  = \frac{\det ( P^Z_{\lambda_j + n - j}(x_i;\xi;\alpha,\beta,\gamma,\delta|q))_{i,j=1}^n}{\Delta(\x_n)}.
\]
We denote by \( M^Z_{\lambda, \mu} \) and \( N^Z_{\lambda, \mu} \)
their multivariate mixed moments and coefficients, respectively:
\begin{align*}
  M^Z_{\lambda, \mu}
  &= M^Z_{\lambda, \mu}(\xi;\alpha,\beta,\gamma,\delta;q)
    = \det(\sigma^Z_{\lambda_i+n-i, \mu_j+n-j}(\xi;\alpha,\beta,\gamma,\delta;q))_{i, j = 1}^n,\\
  N^Z_{\lambda, \mu}
  &= N^Z_{\lambda, \mu}(\xi;\alpha,\beta,\gamma,\delta;q)
    = \det(\nu^Z_{\lambda_i+n-i, \mu_j+n-j}(\xi;\alpha,\beta,\gamma,\delta;q))_{i, j = 1}^n.
\end{align*}
We also denote \( M_\lambda^Z = M_{\lambda, \emptyset}^Z \)
and \( N_\lambda^Z = N_{\lambda, \emptyset}^Z \).
We call \( M^Z_\lambda \) the \emph{Koornwinder moment} (at \( t=q \)).
\end{defn}

By \Cref{thm:Z_nk}, we have the following connection between
Koornwinder moments and ASEP: for $\lambda,\mu\in \Par_n$,
\begin{equation}\label{eq:10}
  M^{Z}_{\lambda, \mu} = \det\left(Z_{\lambda_i+n-i, \mu_j+n-j}\right)_{i, j=1}^n.
\end{equation}

\begin{remark}\label{rem:3}
  Corteel and Williams \cite[Definition~4.2]{Corteel2019} define the Koornwinder
  moment as
\[
  K_\lambda = \frac{\det\left(Z_{\lambda_i+n-i+n-j}\right)_{i, j=1}^n}{\det\left(Z_{2 n-i-j}\right)_{i, j=1}^n}.
\]
By \Cref{lem:6} and \Cref{thm:Z_nk}, we have
\( K_\lambda=M^Z_\lambda \), hence their definition of the Koornwinder
moment is equivalent to ours.
\end{remark}

To state Rains' conjecture precisely, we need the following
definitions.

For a rational function \( f \) and a polynomial \( g \) in some
variables, we say that \( g \) is a \emph{denominator} of \( f \) if
\( fg \) is a polynomial. In this case, we also say that \( fg \) is a
\emph{numerator} of \( f \).
By \Cref{lem:3} and \eqref{eq:10}, \( M^{Z}_{\lambda,\mu} \) has a
denominator consisting of factors of the form
\( (\alpha\beta-q^i\gamma\delta) \), where \( i \) is a nonnegative
integer.  We define the
\emph{minimal denominator} of \( M^{Z}_{\lambda,\mu} \) to be the
lowest-degree polynomial of the form
\( D=\prod_{i\in I} (\alpha\beta-q^i\gamma\delta) \) for a multiset
\( I \) of nonnegative integers such that
\( D\cdot M^{Z}_{\lambda,\mu} \) is a polynomial. We also define the
\emph{minimal numerator} of \( M^{Z}_{\lambda,\mu} \) to be
\( D\cdot M^{Z}_{\lambda,\mu} \).

By \Cref{lem:10}, every \(N^Z_{\lambda,\mu}\) is, up to sign, equal to
\(M^Z_{\eta,\theta}\) for some partitions \(\eta\) and \(\theta\).
Hence, \(N^Z_{\lambda,\mu}\) has the same form of denominator factors,
and we define its minimal denominator and
minimal numerator in the same way.

By \Cref{lem:3}, \(\widetilde{Z}_{n,k}\) is a numerator of \(Z_{n,k}\).
Since this polynomial is a nonzero generating function for rhombic
staircase tableaux with nonnegative coefficients, no factor of the form
\(\alpha\beta-q^i\gamma\delta\) divides it. Indeed, otherwise
\(\widetilde{Z}_{n,k}\) would vanish after substituting
\(\alpha=q^i\gamma\delta/\beta\). Under this substitution, however,
each monomial becomes a Laurent monomial with a nonnegative coefficient,
so no cancellation is possible and the result is still nonzero.
Consequently, the minimal numerator of \(Z_{n,k}\) is \(\widetilde{Z}_{n,k}\),
which is positive by its tableau interpretation.
This suggests a natural extension of the positivity to Koornwinder moments.
Motivated by this, Rains conjectured the positivity of the minimal numerator of the Koornwinder moment
in personal communication with Corteel and Williams \cite[Conjecture 4.4]{Corteel2019}.

\begin{conj}[Rains' conjecture] 
  For a partition \( \lambda \in \Par_n \), the minimal numerator of
  \( M^{Z}_{\lambda}(\xi; \alpha, \beta, \gamma, \delta; q) \) is a polynomial
  with nonnegative integer coefficients.
\end{conj}

To find the minimal numerator of \( M^{Z}_\lambda \), we need to identify its
minimal denominator. By \Cref{lem:3}, we can find the minimal denominator of each
entry \( Z_{\lambda_i+n-i, n-j} \) in the definition of
\( M^{Z}_\lambda \):
\[
  M^{Z}_\lambda = \det\left(Z_{\lambda_i+n-i, n-j}\right)_{i, j=1}^n.
\]
However, in the determinant expansion, there are nontrivial
cancellations of denominators, as the following example shows.

\begin{exam}\label{exa:1}
  Consider
\[
  M^{Z}_{(1,1)} = \det
  \begin{pmatrix}
Z_{2,1} & Z_{2,0}\\
Z_{1,1} & Z_{1,0}
  \end{pmatrix}
  = Z_{2,1} Z_{1,0} - Z_{2,0} Z_{1,1}.
\]
The minimal denominators of \( Z_{2,1} Z_{1,0} \) and
\( Z_{2,0} Z_{1,1} \) are
\( (\alpha\beta-\gamma\delta)(\alpha\beta-q^2\gamma\delta) \) and
\( (\alpha\beta-\gamma\delta)(\alpha\beta-q\gamma\delta) \),
respectively. However, the minimal denominator of \( M^{Z}_{(1,1)} \)
is \( (\alpha\beta-q\gamma\delta)(\alpha\beta-q^2\gamma\delta) \).
\end{exam}

It is a natural question whether Rains' conjecture can be generalized
to \( M^Z_{\lambda, \mu} \). However, in general, the minimal
numerator of \( M^Z_{\lambda, \mu} \) is not positive as illustrated
by the following example.

\begin{exam}\label{rem:6}
  The minimal numerator of \( M^{Z}_{(3,2),(1,0)} \) does not have 
  nonnegative integer coefficients. Substituting
  \( \alpha=\beta=\gamma=\delta=\xi=1 \) into the minimal numerator of
  \( M^{Z}_{(3,2),(1,0)} \) gives
\begin{multline*}
    -8 q^{14} - 112 q^{13} - 600 q^{12} - 1704 q^{11} - 3216 q^{10} -
    4264 q^9 - 4208 q^8 - 2608 q^7 \\
    - 8 q^6 + 2720 q^5 + 4264 q^4 + 4312 q^3 + 3232 q^2 + 1656 q + 544.
\end{multline*}
\end{exam}

Nevertheless, we can generalize Rains' conjecture to
\( M^Z_{\lambda, \mu} \) in some special cases. To do this, we need
some definitions.

A \emph{(0,1)-substitution map} is an operator \( \varphi \) on
\( \QQ(\xi,\alpha,\beta,\gamma,\delta,q) \) that substitutes some
(possibly none) of the parameters
\( \xi,\alpha,\beta,\gamma,\delta,q \) with \( 0 \) or \( 1 \) such
that \( \varphi(\alpha\beta-q^i\gamma\delta) \) is a non-constant
polynomial in \( \xi,\alpha,\beta,\gamma,\delta,q \) for all integers
\( i\ge0 \). We define the \emph{minimal denominator} of
\( \varphi(M^Z_{\lambda, \mu}) \) to be the lowest-degree polynomial
\( \varphi(D) \) such that \( \varphi(D\cdot M^Z_{\lambda, \mu}) \) is
a polynomial in \( \xi,\alpha,\beta,\gamma,\delta,q \), where
\( D=\prod_{i\in I} (\alpha\beta-q^i\gamma\delta) \) for a multiset
\( I \) of nonnegative integers. The \emph{minimal numerator} of
\( \varphi(M^Z_{\lambda, \mu}) \) is defined to be
\( \varphi(D\cdot M^Z_{\lambda, \mu}) \). Note that the
well-definedness of the minimal denominator of
\( \varphi(M^Z_{\lambda, \mu}) \) follows from the assumption that
\( \varphi(\alpha\beta-q^i\gamma\delta) \) is a non-constant
polynomial for all \( i\ge0 \).

\begin{conj}[Generalized Rains' conjecture]\label{con:gen_rains}
  Let \( \varphi \) be a \( (0,1) \)-substitution map such that 
  at least one of the following conditions holds:
\[
\varphi(q)=1,\qquad
\varphi(q)=0,\qquad
\varphi(\alpha)=0,\qquad
\varphi(\beta)=0,\qquad
\varphi(\gamma)=0,\qquad
\varphi(\delta)=0.
\]
For partitions \( \lambda,\mu\in\Par_n \), the minimal numerator
of \( \varphi(M^{Z}_{\lambda,\mu}) \) is a polynomial with nonnegative
integer coefficients.
\end{conj}

The following lemma shows that we do not need to find the minimal
numerator of \( \varphi(M^{Z}_{\lambda,\mu}) \) in order to prove
the generalized Rains' conjecture when \( \varphi(\alpha\beta)\ne0 \). Instead, it suffices to show that
\( \varphi(D\cdot M^{Z}_{\lambda,\mu}) \) is a polynomial with
nonnegative coefficients for some \( D \) of the form
\( \prod_{i\in I}(\alpha\beta-q^i\gamma\delta) \). The proof of the
lemma will be given in \Cref{sec:minimal-denominators}.

\begin{lem}\label{lem:9}
  Let \( \varphi \) be a \( (0,1) \)-substitution map with
  \( \varphi(\alpha\beta)\ne 0 \). Let
  \(D_{\min}=\prod_{i\in I_{\min}} (\alpha\beta-q^i\gamma\delta)\),
  where \( I_{\min} \) is a multiset of nonnegative integers, such
  that \(\varphi(D_{\min})\) is the minimal denominator of
  \(\varphi(M^Z_{\lambda,\mu})\). Then
  \( \varphi(D_{\min} M^{Z}_{\lambda,\mu}) \) has nonnegative
  integer coefficients if and only if there exists
  \( D=\prod_{i\in I}(\alpha\beta-q^i\gamma\delta) \) for some
  multiset \( I \) of nonnegative integers such that
  \(\varphi(DM^Z_{\lambda,\mu})\) is a polynomial with nonnegative
  integer coefficients.
\end{lem}

We also conjecture the following positivity involving the coefficients
\( N^{Z}_{\lambda} \) of the rescaled multivariate Askey--Wilson
polynomials.

\begin{conj}
  For a partition \( \lambda \in \Par_n \), the minimal numerator of
  \( (-1)^{|\lambda|}N^{Z}_{\lambda}(\xi; \alpha, \beta, \gamma, \delta; q) \)
  is a polynomial with nonnegative integer coefficients.
\end{conj}

\subsection{Explicit formulas for Koornwinder moments} \label{sec:deno}

We begin with necessary definitions.
For a given sequence \( f = (f_0,f_1,\dots) \), we write
\[
  (x|f)^n := \prod_{k=0}^{n-1} (x - f_k).
\]
The \emph{complete homogeneous symmetric polynomial}
\( h_n(x_1,\dots,x_{k}) \) is defined by
\[
  h_n(x_1 ,\dots, x_k) = \sum_{1 \le i_1 \le \cdots \le i_n \le k}
  x_{i_1} \cdots x_{i_n}.
\]
The following elementary result will be used later.

\begin{lem}
For a given sequence \( f = (f_0,f_1,\dots) \), we have
\begin{equation}\label{eq:xf-exp}
  x^n = \sum_{k=0}^n h_{n-k}(f_0 ,\dots, f_k) (x|f)^k.
\end{equation}
\end{lem}

\begin{proof}
Since \((x|f)^k\) has degree \(k\), we may write
\( x^n = \sum_{k=0}^{n} c_{n,k}\,(x|f)^k \) and set
\(c_{n,k}=0\) if \(k<0\) or \(k>n\). Since
\( x\cdot(x|f)^k = (x|f)^{k+1} + f_k\,(x|f)^k \),
we have
\[
  c_{n+1,k} = c_{n,k-1} + f_k\,c_{n,k},
  \qquad c_{0,0}=1.
\]
On the other hand, the complete homogeneous symmetric polynomials
satisfy
\[
  h_{n+1-k}(f_0,\dots,f_k)
  =
  h_{n+1-k}(f_0,\dots,f_{k-1})
  + f_k\,h_{n-k}(f_0,\dots,f_k),
\]
with \(h_0=1\) and \(h_m=0\) for \(m<0\). Hence
\(c_{n,k}=h_{n-k}(f_0,\dots,f_k)\) for all \(n,k\), proving
\eqref{eq:xf-exp}.
\end{proof}

For a given sequence \( f = (f_0,f_1,\dots) \) and partitions
\( \lambda, \mu \in \Par_n \), we define
\[
 H_{\lambda, \mu}(f) := \det \left(
      h_{\lambda_i - \mu_j + j - i}(f_0 ,\dots, f_{\mu_j + n - j}) \right)_{i, j = 1}^{n}.
\]

\begin{remark}
  If \( \x=(x_0,x_1,\dots) \) is a sequence of variables,
  \( H_{\lambda, \mu}(\x) \) coincides with a flagged Schur
  polynomial, which is a generating function for semistandard Young
  tableaux with bounds on the entries. See \cite{Wachs1985} for more
  details.
\end{remark}

Now we are ready to state the main result of this subsection. We
denote by 
\[
  M^{AW}_{\lambda,\mu} = M^{AW}_{\lambda,\mu}(a, b, c, d; q)
\]
 the mixed moments of the
multivariate Askey--Wilson polynomials \(P^{AW}_\lambda(\x_n;a, b, c, d|q)\), and
write \(M^{AW}_\lambda=M^{AW}_{\lambda,\emptyset}\). The following
theorem gives an explicit formula for \( M^{AW}_\lambda \).

\begin{thm} \label{prop:explicit_formula}
  For a partition \( \lambda \in \Par_n \), we have
  \begin{align*}
    M^{AW}_\lambda(a, b, c, d; q) &= \sum_{\mu \subseteq \lambda} (-2a)^{-|\mu|} H_{\lambda , \mu}(f)
    s_\mu(1, q, \dots, q^{n-1}) \\
      &\quad \times \prod_{i=1}^n q^{\binom{n-i}{2}-\binom{\mu_i+n-i}{2}} \frac
      {(ab q^{n-i},ac q^{n-i},ad q^{n-i};q)_{\mu_i}}
      {\left(abcd q^{2 n-i-1};q \right)_{\mu_i}},
  \end{align*}
  where \( f = (f_0,f_1,\dots) \) is the sequence given by \( f_k = (aq^k + a^{-1}q^{-k})/2 \).
\end{thm}

To prove \Cref{prop:explicit_formula}, we recall three ingredients.
The first is a product formula for multivariate little \(q\)-Jacobi
moments. The monic \emph{little \( q \)-Jacobi polynomials} \( P_n^L(x) \) and
the monic \emph{big \( q \)-Jacobi polynomials}
\( P_n^B(x) \) are defined by
\begin{align*}
  P_n^L(x)&=P_n^L(x; a, b; q) = \frac{(aq; q)_n}{(-1)^n q^{- \binom{n}{2}}(abq^{n+1};q)_n}
                      \qHyper{2}{1}{q^{-n}, abq^{n+1}}{aq}{q, qx}, \\
  P_n^B(x)&=P_n^B(x; a, b, c; q) =
                         \frac{(aq, cq; q)_n}{(abq^{n+1};q)_n}
                         \qHyper{3}{2}{q^{-n}, abq^{n+1}, x}{aq, cq}{q, q}.
\end{align*}
We refer the reader to \cite{GR} for the definition of
\( q \)-hypergeometric series. We note that both little and big
\( q \)-Jacobi polynomials are classical families of orthogonal
polynomials; see \cite{KLS} for a comprehensive treatment.
Let \( \sigma_{n, k}^L = \sigma_{n, k}^L(a,b;q) \) denote the mixed moment of the little \( q \)-Jacobi polynomials:
\[
  x^n = \sum_{k=0}^n \sigma_{n, k}^LP_k^L(x).
\]
Using the \( q \)-Selberg integral,
Corteel and Kim~\cite[Theorem~5.1]{LHT}
showed the following product formula for
the multivariate moment of little \( q \)-Jacobi polynomials:
for \( \lambda\in\Par_n \),
\begin{equation}\label{thm:LHT_M_N}
 \det\left(\sigma^L_{\lambda_i+n-i, n-j}\right)_{i, j=1}^n= s_\lambda(1, q ,\dots, q^{n-1})
  \prod_{i=1}^n \frac{\left(a q^{n-i+1}; q\right)_{\lambda_i}}
  {\left(a b q^{2 n-i+1}; q\right)_{\lambda_i}}.
\end{equation}

The second ingredient is the following result from
\cite[Propositions~4.17 and 4.21]{Corteel2023}.
\begin{prop}\label{thm:L_B_AW} 
  Let \( \widetilde{\sigma}_{n, k}^B = \widetilde{\sigma}_{n, k}^B(a, b, c ; q) \) and
  \( \widetilde{\sigma}_{n, k}^{AW}=\widetilde{\sigma}_{n, k}^{A W}(a, b, c, d ; q) \) be defined by
  \begin{align*}
    \prod_{i=0}^{n-1}(x - q^{-i}) = \sum_{k=0}^n \widetilde{\sigma}_{n, k}^BP_k^B(x), \qquad
    (x|f)^n =  \sum_{k=0}^n \widetilde{\sigma}_{n, k}^{AW}P_k^{AW}(x),
  \end{align*}
    where \( f = (f_0,f_1,\dots) \) is the sequence given by \( f_k = (aq^k + a^{-1}q^{-k})/2 \).
  Then we have
  \begin{align}
    \label{eq: tilde B}\widetilde{\sigma}_{n, k}^B(a, b, c ; q)&=(-1)^{n-k} q^{\binom{k}{2}-\binom{n}{2}}\left(c q^{k+1} ; q\right)_{n-k} \sigma_{n, k}^L(a, b ; q), \\
    \label{eq: tilde AW}\widetilde{\sigma}_{n, k}^{A W}(a, b, c, d ; q) & =(2 a)^{k-n}\left(a d q^k ; q\right)_{n-k} \widetilde{\sigma}_{n, k}^B(a b / q, c d / q, a c / q ; q).
  \end{align}
\end{prop}

The final ingredient is the Cauchy--Binet formula.

\begin{lem}[Cauchy--Binet] \label{lem:CB}
  Let \( A \) be an \( i \times k \) matrix and let \( B \) be a \( k \times j \) matrix.
  For subsets \(I\subseteq\{1,\dots,i\}\) and
  \(J\subseteq\{1,\dots,j\}\) with \(|I|=|J|=m\le k\), we have
  \[
    [AB]_{I, J} = \sum_{K} [A]_{I, K}[B]_{K, J},
  \]
  where the sum is over all subsets \( K \) of \( \{1 ,\dots, k\} \) of size \( m \).
\end{lem}

We now combine these ingredients to prove \Cref{prop:explicit_formula}.

\begin{proof}[Proof of \Cref{prop:explicit_formula}]
  We use the notation in \Cref{thm:L_B_AW}. Let
  \begin{align*}
    M^{L} &= M^{L}(a,b;q) = (\sigma_{s, t}^L(a,b;q))_{s,t\ge0},\\
    \widetilde{M}^{B} &=  \widetilde{M}^{B}(a,b,c;q) = (\widetilde{\sigma}_{s, t}^B(a,b,c;q))_{s,t\ge0},\\
    \widetilde{M}^{AW} &= \widetilde{M}^{AW}(a,b,c,d;q) = (\widetilde{\sigma}_{s, t}^{AW}(a,b,c,d;q))_{s,t\ge0}.
  \end{align*}
  For finite sets \( I \) and \( J \) of nonnegative integers with
  \( |I|=|J| \), we have
\begin{align*}
  [\widetilde{M}^B]_{I, J}
  &= \det(\widetilde{\sigma}_{i, j}^{B})_{i \in I, j \in J} \\
  &= \det\left((-1)^{i-j} q^{\binom{j}{2}-\binom{i}{2}}\left(c q^{j+1} ; q\right)_{i-j} \sigma_{i, j}^L\right)_{i \in I, j \in J}  \qquad \text{(by \eqref{eq: tilde B})} \\
  &= \det \left( \frac{(-1)^iq^{-\binom{i}{2}} (cq; q)_{i}}
    {(-1)^j q^{-\binom{j}{2}}(cq; q)_{j}}
    \sigma_{i, j}^L \right)_{i \in I, j \in J} \\
  &= \frac{\prod_{i \in I}(-1)^iq^{-\binom{i}{2}} (cq; q)_{i}}
    {\prod_{j \in J}(-1)^j q^{-\binom{j}{2}}(cq; q)_{j}}
    [M^L]_{I, J}.
\end{align*}
Similarly,  using \eqref{eq: tilde AW} in place of \eqref{eq: tilde B}, we have
\begin{align*}
  [\widetilde{M}^{AW}]_{I, J}
  &= \frac{\prod_{i \in I}(2a)^{-i} (ad; q)_{i}}
    {\prod_{j \in J}(2a)^{-j} (ad; q)_{j}} 
    [\widetilde{M}^B(a b / q, c d / q, a c / q ; q)]_{I, J}.
\end{align*}
Combining the above two equations, we obtain
\begin{equation}\label{eq:9}
  [\widetilde{M}^{AW}]_{I, J} = \frac{\prod_{i \in I}q^{-\binom{i}{2}}(-2a)^{-i}(ac; q)_{i} (ad; q)_{i}}
  {\prod_{j \in J} q^{-\binom{j}{2}}(-2a)^{-j}(ac; q)_{j} (ad; q)_{j}}
  [M^L(a b / q, c d / q ; q)]_{I, J}.
\end{equation}

Let \( \widetilde M^{AW}_\lambda = \det\left(
\widetilde\sigma^{AW}_{\lambda_i+n-i,n-j} \right)_{i,j=1}^n \).
Applying \eqref{eq:9} with \(I=\{\lambda_i+n-i : 1 \le i \le n\}\)
and \(J=\{n-i : 1 \le i \le n\}\), and then
applying \eqref{thm:LHT_M_N} with the parameters \( (a,b) \)
replaced by \((ab/q,cd/q)\), we obtain
\begin{equation}\label{eq:8}
  \widetilde{M}^{AW}_{\lambda}\\
  = s_\lambda(1, q, \dots, q^{n-1})
    \prod_{i=1}^n (-2a)^{-\lambda_i} q^{\binom{n-i}{2}-\binom{\lambda_i+n-i}{2}} \frac
    {(ab q^{n-i},ac q^{n-i},ad q^{n-i};q)_{\lambda_i}}
    {\left(abcd q^{2 n-i-1};q \right)_{\lambda_i}} .
\end{equation}
From the definition of \(\widetilde{\sigma}_{s, k}^{AW}\), we have
\((x|f)^s = \sum_{k=0}^s \widetilde{\sigma}_{s, k}^{AW}P_k^{AW}(x)\), and
by \eqref{eq:xf-exp} we get
\[
  x^s = \sum_{t=0}^s\sum_{k=0}^t
  h_{s-t}(f_0 ,\dots, f_t) \widetilde{\sigma}_{t, k}^{AW}P_k^{AW}(x).
\]
Thus, the mixed moment \( \sigma^{AW}_{s, k} \) of the OPS \( \{P_k^{AW}(x)\} \) is given by
\[
  \sigma^{AW}_{s, k} = \sum_{t=0}^s h_{s-t}(f_0 ,\dots, f_t) \widetilde{\sigma}_{t, k}^{AW}.
\]
Rewriting this formula in terms of matrices, we obtain
\( M^{AW}=H \widetilde{M}^{AW} \), where
\( M^{AW} = (\sigma_{s, t}^{AW})_{s,t\ge0} \) and
\( H = \bigl(h_{s-t}(f_0,\dots,f_t)\bigr)_{s,t\ge0} \). Then, by
\Cref{lem:CB}, we have
\[
  M_{\lambda}^{AW}
  = \sum_{\mu \subseteq \lambda} H_{\lambda , \mu}(f) \widetilde{M}_{\mu}^{AW}.
\]
This, together with \eqref{eq:8}, completes the proof.
\end{proof}

Using \Cref{prop:explicit_formula}, we obtain an explicit formula for \( M^{Z}_\lambda \).

\begin{thm} \label{cor:explicit_Koornwinder_moment}
  For a partition \( \lambda \in \Par_n \),
  we have
  \begin{align*}
    M^{Z}_{\lambda}
    &= \left( \frac{2\sqrt{\xi}}{1-q} \right)^{\left| \lambda \right|}
      \sum_{\mu \subseteq \lambda}
      \left(\frac{-2a}{\sqrt{\xi}}\right)^{-|\mu|} H_{\lambda , \mu}(\bar{f})
    s_\mu(1, q, \dots, q^{n-1}) \\
      &\quad \times \prod_{i=1}^n q^{\binom{n-i}{2}-\binom{\mu_i+n-i}{2}} \frac
      {(ab q^{n-i},ac\xi^{-1}q^{n-i},ad q^{n-i};q)_{\mu_i}}
      {\left(abcd q^{2 n-i-1};q \right)_{\mu_i}},
  \end{align*}
  where \( \bar{f} = (\bar{f}_0,\bar{f}_1,\dots) \) is the sequence given by 
  \[
    \bar{f}_k =
    \frac{a q^{k}+a^{-1}q^{-k}\xi + 1 + \xi}{2\sqrt{\xi}}
  \]
  and the parameters \( (a, b, c, d) \) are substituted by
  \( (\alpha,\beta,\gamma,\delta) \) as in \Cref{def:alpha}.
\end{thm}

\begin{proof}
Let
  \[
    C=\frac{2\sqrt{\xi}}{1-q},
    \qquad
    r=\frac{1+\xi}{2\sqrt{\xi}},
    \qquad 
    Q_n(x) = C^n P_n^{AW}\left(\frac{x}{C}-r \right).
  \]
  Let \( \sigma^Q_{n,k} \) and \( M^Q_{\lambda} \) be the mixed moment
  and the multivariate moment of the OPS \( \{Q_n(x)\}_{n\ge0} \),
  respectively. Then, by \eqref{eq:monic_Z} we have
  \[
    P_{n}^{Z}(x)=Q_n(x)|_{(a,b,c,d) = (a/\sqrt{\xi},b\sqrt{\xi},c/\sqrt{\xi},d\sqrt{\xi})},
  \]
  and therefore,
  \begin{equation}\label{eq:20}
    M^Z_\lambda=M^Q_\lambda|_{(a,b,c,d) =
      (a/\sqrt{\xi},b\sqrt{\xi},c/\sqrt{\xi},d\sqrt{\xi})}.
  \end{equation}

  Let \( f = (f_0,f_1,\dots) \) be the sequence given by
  \( f_k = (aq^k + a^{-1}q^{-k})/2 \) and let
  \( f+r = (f_0+r,f_1+r,\dots) \). Since \( (x+r|f+r)^s = (x|f)^s \),
  using the definition of \(\widetilde{\sigma}_{s, k}^{AW}\) in
  \Cref{thm:L_B_AW}, we have
  \[
    (x+r|f+r)^s = (x|f)^s = \sum_{k=0}^s \widetilde{\sigma}_{s, k}^{AW}
    P_k^{AW}(x).
  \]
  By \eqref{eq:xf-exp} with \( x \) and \( f \) replaced by \( x+r \)
  and \( f+r \), respectively, we have
  \begin{align*}
      (x+r)^n &= \sum_{s=0}^n h_{n-s}(f_0+r ,\dots, f_s+r) (x+r|f+r)^s\\
    &= \sum_{s=0}^n\sum_{k=0}^s
    h_{n-s}(f_0+r ,\dots, f_s+r) \widetilde{\sigma}_{s, k}^{AW}P_k^{AW}(x).
  \end{align*}
  Replacing \( x \) by \( \frac{x}{C} -r \), we obtain
  \[
    \frac{x^n}{C^n} = \sum_{s=0}^n\sum_{k=0}^s
    h_{n-s}(f_0+r ,\dots, f_s+r) \widetilde{\sigma}_{s, k}^{AW}
    P_k^{AW}\left(\frac{x}{C}-r\right).
  \]
  Hence,
\[
    x^n = \sum_{s=0}^n\sum_{k=0}^s
    h_{n-s}(f_0+r ,\dots, f_s+r) \widetilde{\sigma}_{s, k}^{AW}
    C^{n-k}Q_k(x),
  \]
  which implies
  \[
    \sigma^Q_{n,k} = C^{n-k} \sum_{s=k}^n
    h_{n-s}(f_0+r ,\dots, f_s+r) \widetilde{\sigma}_{s, k}^{AW}.
  \]
  By the same argument in the proof of \Cref{prop:explicit_formula}
  with the above identity, we obtain
  \begin{align*} 
    M_\lambda^Q
    &= C^{|\lambda|}\sum_{\mu \subseteq \lambda} (-2a)^{-|\mu|} H_{\lambda , \mu}(f+r)
    s_\mu(1, q, \dots, q^{n-1}) \\
    & \quad\times \prod_{i=1}^n q^{\binom{n-i}{2}-\binom{\mu_i+n-i}{2}} \frac
    {(ab q^{n-i},ac q^{n-i},ad q^{n-i};q)_{\mu_i}}
    {\left(abcd q^{2 n-i-1};q \right)_{\mu_i}}.
  \end{align*}
  Finally, by \eqref{eq:20}, we obtain the desired formula for
  \( M_{\lambda}^Z \).
\end{proof}

By \Cref{def:alpha}, we have \( abcd = (\gamma\delta) / (\alpha\beta) \)
and \( 1-abcdq^i = (\alpha\beta - q^i\gamma\delta) / (\alpha\beta) \).
Therefore, as an immediate consequence of \Cref{cor:explicit_Koornwinder_moment},
we obtain a numerator of \( M^{Z}_\lambda \).

\begin{cor}\label{prop:deno}
  For a partition \( \lambda \in \Par_n \), the following is a polynomial:
  \begin{equation} \label{eq:deno}
   M^{Z}_\lambda \prod_{(i, j) \in \lambda}(\alpha\beta - q^{2n-2-i+j}\gamma\delta).
  \end{equation}
\end{cor}

In \Cref{prop:den}, we will show that the product in \eqref{eq:deno} is
the minimal denominator of \( M^{Z}_\lambda \). 
After treating \(M^Z_\lambda\),  it is natural to consider its dual \(N^Z_\lambda\) as well.
We propose the following conjectural formula for its minimal denominator.

\begin{conj}\label{con:N}
  For \( \lambda\in\Par_n \), the minimal denominator of \( N_\lambda^Z \) is given by
  \[
    \prod_{(i, j) \in \lambda} \left(\alpha \beta - q^{\lambda_i - \lambda'_j+2n-2 - i+ j}\gamma \delta\right).
  \]
\end{conj}

This conjecture has been verified by computer for all partitions
\( \lambda \subseteq (4, 4, 4) \) and for several values of \( n \).

\section{Positivity results in two special cases}
\label{sec:proof-special-cases}

In this section, we establish the positivity of certain numerators of
the Koornwinder moments \( M^{Z}_{\lambda,\mu} \) in the
specializations \((\xi,q)=(1,0)\) and \((\xi,q)=(1,1)\). These results
will be used in the next section to prove the generalized Rains'
conjecture in these cases. In \Cref{subsec:q0-xi1}, we treat the
specialization \((\xi,q)=(1,0)\) by constructing a lattice path model
for the coefficients \(\nu^Z_{n,k}\) of the rescaled Askey--Wilson
polynomials. We then apply the Lindstr\"om--Gessel--Viennot (LGV)
lemma to prove the positivity of the numerator of the multivariate
coefficients \( N^{Z}_{\lambda,\mu} \). By duality, we deduce the
positivity of the numerator of \( M^{Z}_{\lambda,\mu} \). In
\Cref{subsec:q1-xi1}, we use mixed moments of Laguerre polynomials to
derive explicit formulas for \( \sigma^{Z}_{n,k} \) and
\( M^{Z}_{\lambda,\mu} \), and then give a combinatorial
interpretation of \( M^{Z}_{\lambda,\mu} \) in terms of lecture hall tableaux.

\subsection{The case \( (\xi,q)=(1,0) \)}
\label{subsec:q0-xi1}

Throughout this subsection, we assume \( (\xi,q)=(1,0) \). For
simplicity, we write
\( \nu^Z_{n,k} = \nu^Z_{n,k}(1;\alpha,\beta,\gamma,\delta;0) \)
and 
\( N^Z_{\lambda,\mu} = N^Z_{\lambda,\mu}(1;\alpha,\beta,\gamma,\delta;0) \).

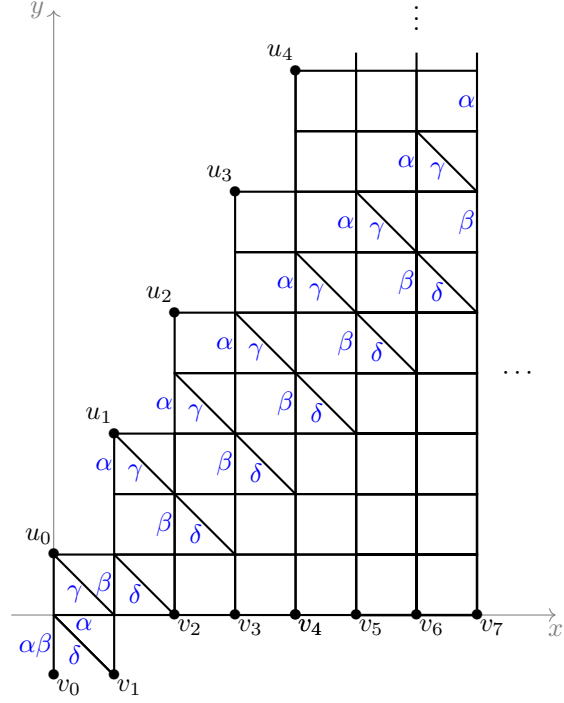
\begin{figure}
  \centering
  \begin{tikzpicture}[x=0.8cm,y=0.8cm, line width=0.8pt]
  \def\N{4}
  \def\Xcut{7}

  \pgfmathsetmacro{\YtopLine}{2*\N+2}
  \pgfmathsetmacro{\Ytopline}{2*\N}
  \draw[gray,thin,->] (-.7,1) -- (\Xcut+1.3,1) node[below] {$x$};
  \draw[gray,thin,->] (0,0) -- (0,\YtopLine+1) node[left] {$y$};
  \node at (\Xcut-1,\YtopLine+1) {\( \vdots \)}; 
  \foreach \k in {0,...,\N} {
    \pgfmathtruncatemacro{\xL}{\N-\k}
    \pgfmathtruncatemacro{\yBot}{2*(\N-\k)}
    \pgfmathtruncatemacro{\yTop}{2*(\N-\k)+1}
    
    \foreach \j in {0,...,\k} {
      \pgfmathtruncatemacro{\x}{\xL+\j}
      \ifnum\x<\Xcut
        \ifnum\k=\N
          \draw (\x,\yTop) rectangle ++(1,1);
        \else
          \draw (\x,\yTop) rectangle ++(1,1);
          \draw (\x,\yBot) rectangle ++(1,1);
        \fi
        \ifnum\yTop>\Ytopline
          \draw (\N+1,\yBot+1) rectangle ++(1,1);
          \draw (\N+2,\yBot+1) rectangle ++(1,1);
          \draw (\N+1,\yTop+1) -- (\N+1,\yTop+1.3);
          \draw (\N+2,\yTop+1) -- (\N+2,\yTop+1.3);
          \draw (\N+3,\yTop+1) -- (\N+3,\yTop+1.3);
        \else  
          \draw (\N+1,\yTop+1) rectangle ++(1,1);
          \draw (\N+1,\yBot+1) rectangle ++(1,1);
          \draw (\N+2,\yTop+1) rectangle ++(1,1);
          \draw (\N+2,\yBot+1) rectangle ++(1,1);
        \fi
      \fi
    }}
  \foreach \k in {0,...,\Xcut} {
    \ifnum\k<\Xcut
      \draw (\k,\k+1) -- (\k+1,\k);
      \node[blue] at (\k+.35,\k+.35) {\( \delta \)};
      \node[blue] at (\k+0.83,\k+3.5) {\( \alpha \)};
      \node[blue] at (\k+0.83,\k+1.5) {\( \beta \)};
    \fi
    
    \ifnum\k>1
      \node at (\k,1) {\( \bullet \)};
      \node at (\k+.25,0.8) {\( v_{\k} \)};
      \draw (\k-1,\k+2) -- (\k,\k+1);
      \node[blue] at (\k+0.35-1,\k+2.35-1) {\( \gamma \)};
    \fi
    \ifnum\k<\N
      \node at (\k,2*\k+2) {\( \bullet \)};
      \node at (\k-0.25,2*\k+2.3) {\( u_\k \)};
    \fi

    \ifnum\k=0
      \node at (0,0) {\( \bullet \)};
      \node at (0.25,-0.2) {\( v_0 \)};
      \node at (1,0) {\( \bullet \)};
      \node at (1.25,-0.2) {\( v_1 \)};
      \draw (0,2) -- (1,1);
      \node[blue] at (0.35,1.35) {\( \gamma \)};
    \else
      \ifnum\k<7
      \fi
    \fi
  }

  \node[blue] at (0.5,0.85) {\( \alpha \)};
  \node[blue] at (-0.3,0.5) {\( \alpha\beta \)};
  
  \draw (0,1) -- (0,0);
  \draw (1,1) -- (1,0);
  \draw (0,1) -- (1,0);

  \node at (4,1) {\( \bullet \)};
  \node at (4.25,0.8) {\( v_4 \)};
  \draw (4,8) -- (4,10);
  \node at (4,10) {\( \bullet \)};
  \node at (3.75,10.3) {\( u_4 \)};

  \node[right] at (\Xcut+.25,5.0) {\(\cdots\)};
\end{tikzpicture}
\caption{The weighted lattice model \( G \).}
  \label{fig:HJlattice}
\end{figure}

We first define the directed lattice graph used in the path model; see
\Cref{fig:HJlattice} for an illustration. Let \(G=(\mathcal{V},\mathcal{E})\) be the directed
graph with vertex set
\[
\mathcal{V}=\{(0,-1),(1,-1)\}\cup
\{(i,j)\in \mathbb Z_{\ge0}\times \mathbb Z_{\ge0}: 0\le j\le 2i+1\},
\]
and edge set \(\mathcal{E}=\mathcal{E}_{E}\cup \mathcal{E}_S\cup \mathcal{E}_D \), where
\[
\begin{aligned}
\mathcal{E}_E
&=\{\, (i,j)\to(i+1,j)
      : j\ge0,\ (i,j),(i+1,j)\in \mathcal{V} \,\},\\
\mathcal{E}_S
&=\{\, (i,j)\to(i,j-1)
      : (i,j),(i,j-1)\in \mathcal{V} \,\},\\
\mathcal{E}_D
&=\{\, (i,j)\to(i+1,j-1)
      : (i,j)=(0,1),\ \text{or } j=i \text{ with } i\ge0,\\
&\hspace{12.8em}
      \text{or } j=i+2 \text{ with } i\ge1 \,\}.
\end{aligned}
\]
The directed edges of \(G\) are the line segments shown in
\Cref{fig:HJlattice}, oriented from left to right or from top to
bottom. We refer to the elements of \(\mathcal{E}_E\), \(\mathcal{E}_S\), and \(\mathcal{E}_D\) as
\emph{east}, \emph{south}, and \emph{diagonal} steps, respectively.
The \emph{weight} \( \wt(A) \) of a step \( A\in\mathcal{E} \)
with an initial point \((i,j)\) is defined as follows:
\begin{align*}  
  \wt((i,j)\to(i+1,j)) &=
  \begin{cases}
    \alpha & \text{if } (i,j) = (0,0), \\   
    0 & \text{if } (i,j) = (0,-1), \\
    1 & \text{otherwise},
  \end{cases}
  \\
  \wt((i,j)\to(i,j-1)) &= 
  \begin{cases}
    \alpha & \text{if } j=i+2 \text{ with } i\geq1, \\
    \beta & \text{if } j=i \text{ with } i\geq1, \\
    \alpha\beta & \text{if } (i,j) = (0,0), \\
    1 & \text{otherwise},
  \end{cases}
  \\
  \wt((i,j)\to(i+1,j-1)) &= 
  \begin{cases}
    \gamma & \text{if } j=i+2 \text{ with } i\geq1, \text{ or } (i,j)=(0,1), \\
    \delta & \text{if } j=i \text{ with } i\geq0, \\
    0 & \text{otherwise}.
  \end{cases}
\end{align*}
For a path \(p\) in \(G\), its \emph{weight} \(\wt(p)\) is defined to be the
product of the weights of its steps.

Let \(P(u\to v)\) denote the set of (directed) paths from \(u\) to
\(v\). We will mostly consider paths in \(P(u_i\to v_j)\) for
\( i,j\ge0 \), where
\[
u_i=(i,2i+1),\qquad
v_i=
\begin{cases}
(i,-1), & \text{if } i=0,1,\\
(i,0), & \text{if } i\ge2.
\end{cases}
\]

Recall the definitions of \( \nu^Z_{n,k} \) and \( V_{n,k} \) in \Cref{def:GZ} and \eqref{eq:V}.
By \Cref{cor: coe xi=1 q=0}, we have the following explicit formula for
\( \nu^Z_{n,k} = \nu^Z_{n,k}(1;\alpha,\beta,\gamma,\delta;0) \):
\begin{equation}\label{eq:13}
(-1)^{n-k}\nu^Z_{n,k}
= \begin{cases}
 \frac{\alpha + \beta + \gamma + \delta}{\alpha\beta - \gamma \delta} & \mbox{if \( (n,k)=(1,0) \)},\\
 \frac{\alpha + \beta + \gamma+\delta+1}{\alpha\beta} & \mbox{if \( (n,k)=(2,0) \)},\\
\frac{V_{n,k}(1;\alpha,\beta,\gamma,\delta)}{\alpha\beta}& \mbox{otherwise.}
  \end{cases}
\end{equation}
Note that
\begin{align}
  \label{eq:4}
V_{n,k}(1;\alpha,\beta,\gamma,\delta)
&= \binom{n+k-1}{2k-1} \alpha\beta
 + \binom{n+k-2}{2k-1}(\alpha\delta + \beta\gamma)
 + \binom{n+k-3}{2k-1} \gamma\delta \\
  \notag
&\quad + \binom{n+k-1}{2k}(\alpha+\beta)
 + \binom{n+k-2}{2k}(\gamma+\delta)
 + \binom{n+k-1}{2k+1}.
\end{align}
We define \( \widetilde{\nu}^Z_{n,k} \) to be the numerator of \( \nu^Z_{n,k} \)
shown in \eqref{eq:13} without the sign:
\begin{equation}\label{eq:12}
\widetilde{\nu}^Z_{n,k}
= \begin{cases}
\alpha+\beta+\gamma+\delta, & \mbox{if \( (n,k)=(1,0) \)},\\
\alpha+\beta+\gamma+\delta+1, & \mbox{if \( (n,k)=(2,0) \)},\\
V_{n,k}(1;\alpha,\beta,\gamma,\delta), & \mbox{otherwise.}
\end{cases}
\end{equation}
The following proposition gives a path interpretation of these numerators.

\begin{prop}\label{prop:8}
For all \(n,k\ge0\), we have
\[
    \widetilde{\nu}^Z_{n,k}= \sum_{p\in P(u_k\to v_n)}\wt(p). 
\]
\end{prop}
\begin{proof}
Suppose \(k=0\). By \eqref{eq:12}, we have
\[
  \widetilde{\nu}^Z_{n,0}=
  \begin{cases}
    \alpha\beta, & \text{if } n=0,\\
    \alpha+\beta+\gamma+\delta+n-1, & \text{if } n\ge1.
  \end{cases}
\]
By the definition of \( G \),
this is equal to \( \sum_{p\in P(u_0\to v_n)}\wt(p) \).

Now suppose \( k \geq 1 \). Since
\( \widetilde{\nu}^Z_{n,k} = V_{n,k}(1;\alpha,\beta,\gamma,\delta) \),
it suffices to prove the following identity:
\[
  \sum_{p\in P(u_k \to v_n)} \wt(p) = V_{n,k}(1;\alpha,\beta,\gamma,\delta).
\]

Consider \(p\in P(u_k \to v_n)\). Let \( A \) and \( B \) be the
points of \( p \) on the lines \(y=x+2\) and \(y=x-2\), respectively.
Then we decompose \( p \) as \(p=p_1p_2p_3 \), where \(p_1\) is the
initial part from \(u_k\) to \(A\), \(p_2\) is the middle part from
\(A\) to \(B\), and \(p_3\) is the final part from \(B\) to \(v_n\);
see \Cref{fig:pathdecomp}. Since all steps in \(p_1\) and \(p_3\) have
weight \(1\), we have
\[
  \sum_{p\in P(u_k\to v_n)}\wt(p)
  =
  \sum_{p=p_1p_2p_3\in P(u_k\to v_n)}\wt(p_2).
\]
  
  \begin{figure}
    \centering
    \begin{tikzpicture}[x=0.7cm,y=0.7cm, line width=0.8pt]
     \clip (-1.5, -1.5) rectangle (8.5, 9.5); 

      \def\N{6}
      
      \begin{scope}[draw=gray!60, very thin]
        \foreach \k in {0,...,\N} {
          \pgfmathsetmacro{\xL}{\N-\k}
          \pgfmathsetmacro{\yBot}{2*(\N-\k)}
          \pgfmathsetmacro{\yTop}{2*(\N-\k)+1}
          
          \foreach \j in {0,...,\k} {
            \pgfmathsetmacro{\x}{\xL+\j}
            \ifnum\k=\N
              \draw (\x,\yTop) rectangle ++(1,1);
            \else
              \draw (\x,\yTop) rectangle ++(1,1);
              \draw (\x,\yBot) rectangle ++(1,1);
            \fi
          }
          
          \ifnum\k<\N
            \draw (\k+1,\k+2) -- (\k+2,\k+1);
            \ifnum\k=0
              \draw (0,2) -- (1,1);
            \else
              \draw (\k,\k+3) -- (\k+1,\k+2);
            \fi
          \fi
        }
        \draw (0,1) -- (0,0);
        \draw (1,1) -- (1,0);
        \draw (0,1) -- (1,0);
      \end{scope}

      \draw[thick, black!80, dashed] (-0.5, 2.5) -- (6.5, 9.5) node[below right, font=\small, pos=0.85] {\(y=x+2\)};
      \draw[thick, black!80, dashed] (0.5, -0.5) -- (8, 7) node[above left, font=\small, pos=0.85] {\(y=x-2\)};

      \coordinate (U) at (2,6);
      \coordinate (S) at (2,5); 
      \coordinate (E) at (5,4); 
      \coordinate (V) at (7,1); 

      \draw[ultra thick, green!60!black] (U) -- (S) node[midway, left, font=\small] {\(p_1\)};

      \draw[ultra thick,green!60!black] (S) -- (3,5) -- (4,5) -- (E);
      \node[green!60!black, above, font=\small] at (3.5,5) {\(p_2\)};

      \draw[ultra thick,green!60!black] (E) -- (5,3) -- (6,3) -- (6,2) -- (6,1) -- (V);

      \node[green!60!black, font=\small] at (6.3, 2.5) {\(p_3\)};

      \node[circle, fill=black, inner sep=1.5pt, label={[font=\small]above:\(u_k\)}] at (U) {};
      \node[circle, fill=black, inner sep=1.5pt, label={[font=\small, label distance=2pt]-90:\(v_n\)}] at (V) {};
      \node[circle, fill=black, inner sep=1.5pt, label={[black, font=\small]left:\(A\)}] at (S) {};
      \node[circle, fill=black, inner sep=1.5pt, label={[black, font=\small]right:\(B\)}] at (E) {};

    \end{tikzpicture}%
    \caption{The path decomposition \( p=p_1p_2p_3 \) of
      a path \( p\in P(u_k\to v_n) \) with \( k=2 \) and \( n=7 \).}
    \label{fig:pathdecomp}
  \end{figure}
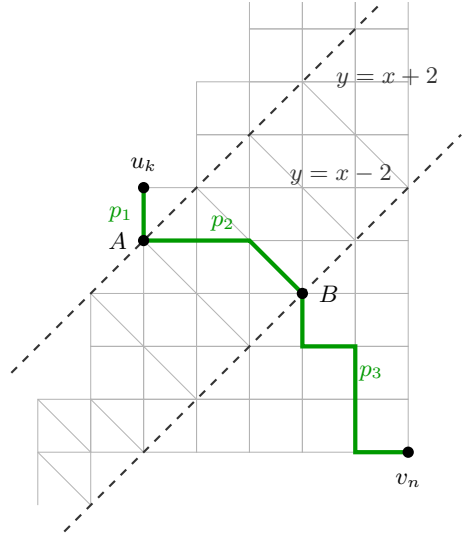
 
  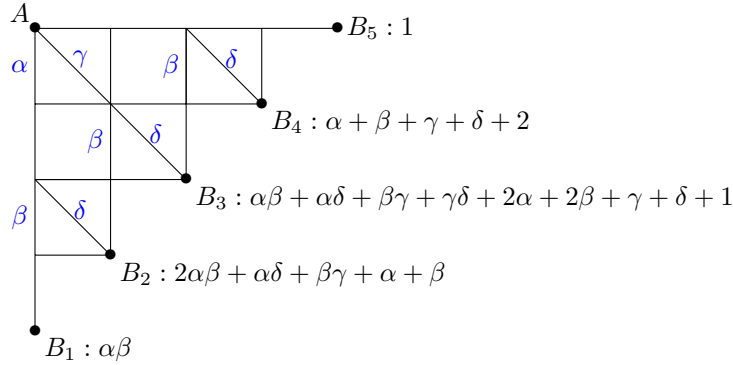
\begin{figure}
    \centering
    \begin{tikzpicture}
      \draw (0,0) grid (2,2);
      \draw (0,0) grid (1,-1);
      \draw (2,2) grid (3,1);
      \draw (0,2) -- (2,0);
      \draw (0,0) -- (1,-1);
      \draw (2,2) -- (3,1);
      \draw (0,-1) -- (0,-2);
      \draw (3,2) -- (4,2);
      
      \node[blue] at (-.2,1.5) {\( \alpha \)};
      \node[blue] at (-.2,-.5) {\( \beta \)};
      \node[blue] at (.8,.5) {\( \beta \)};
      \node[blue] at (1.8,1.5) {\( \beta \)};
      \node[blue] at (.6,1.6) {\( \gamma \)};
      \node[blue] at (.6,-.4) {\( \delta \)};
      \node[blue] at (1.6,.6) {\( \delta \)};
      \node[blue] at (2.6,1.6) {\( \delta \)};
      
      \node at (0,-2) {\( \bullet \)};
      \node[anchor=west] at (0,-2.25) {\( B_1:\alpha\beta \)};
      \node at (1,-1) {\( \bullet \)};
      \node[anchor=west] at (1,-1.25) {\( B_2:2\alpha\beta+\alpha\delta+\beta\gamma+\alpha+\beta \)};
      \node at (2,0) {\( \bullet \)};
      \node[anchor=west] at (2,-.25) {\( B_3:\alpha\beta+\alpha\delta+\beta\gamma+\gamma\delta+2\alpha+2\beta+\gamma+\delta+1 \)};
      \node at (3,1) {\( \bullet \)};
      \node[anchor=west] at (3,0.75) {\( B_4:\alpha+\beta+\gamma+\delta+2 \)};
      \node at (4,2) {\( \bullet \)};
      \node[anchor=west] at (4,2) {\( B_5:1 \)};

      \node at (0,2) {\( \bullet \)};
      \node at (-.2,2.2) {\( A \)};
    \end{tikzpicture}%
    \caption{The possible end points \( B_i \) of the middle part \( p_2 \) starting from a fixed vertex \( A \).
      The sum of the weights of all paths \( p_2\in P(A,B_i) \) is shown next to each \( B_i \).}
    \label{fig:HJlatticeproof}
  \end{figure}

Once the point \(A=(r,r+2)\) is fixed, the possible points \(B\) are
\[
B_i=(r+i-1,r+i-3),\qquad 1\le i\le5.
\]
The weight sum of the possible middle parts \( p_2 \) ending at each point is
as follows:
\[
  \sum_{p_2\in P(A\to B_i)}\wt(p_2)
  = \begin{cases}
 \alpha\beta & \mbox{if \( i=1 \),}\\
 2\alpha\beta+\alpha\delta+\beta\gamma+\alpha+\beta  & \mbox{if \( i=2 \),}\\
 \alpha\beta+\alpha\delta+\beta\gamma+\gamma\delta+
 2\alpha+2\beta+\gamma+\delta+1  & \mbox{if \( i=3 \),}\\
 \alpha+\beta+\gamma+\delta+2  & \mbox{if \( i=4 \),}\\
 1  & \mbox{if \( i=5 \)}.
  \end{cases}
\]
See \Cref{fig:HJlatticeproof} for the corresponding local picture.

For fixed \(A=(r,r+2)\) and \( B_i \), the numbers of possible choices for \( p_1 \) and \( p_3 \) are \( \binom{k-1}{r-k} \) and \( \binom{n-2}{r+i-3} \), respectively. Hence, summing over all
possible positions of \(A\), the number of pairs \((p_1,p_3)\) for which
\(p_2\) ends at \(B_i\) is
\[
\sum_{r\ge0}
\binom{k-1}{r-k}
\binom{n-2}{r+i-3}
=
\binom{n+k-3}{2k+i-4}.
\]
Therefore, we obtain
\begin{align*}
\sum_{p\in P(u_k \to v_n)} \wt(p) &= \binom{n+k-3}{2k-3}\alpha\beta 
 + \binom{n+k-3}{2k-2}(\alpha\delta+\beta\gamma+2\alpha\beta+\alpha+\beta) \\
&\quad + \binom{n+k-3}{2k-1}(\gamma\delta+\alpha\delta+\beta\gamma+\alpha\beta+2\alpha+2\beta+\gamma+\delta+1) \\
                                  &\quad + \binom{n+k-3}{2k}(\alpha+\beta+\gamma+\delta+2) + \binom{n+k-3}{2k+1}\\
  &= V_{n,k}(1;\alpha,\beta,\gamma,\delta),
\end{align*}
as desired.
\end{proof}
For \(\lambda,\mu\in\Par_n\), we define
\begin{equation}\label{eq:11}
\widetilde N^Z_{\lambda,\mu}
=
\det\left(\widetilde\nu^Z_{\lambda_i+n-i,\mu_j+n-j}\right)_{i,j=1}^n .
\end{equation}
Using the path interpretation in \Cref{prop:8}, we now prove that
\(\widetilde N^Z_{\lambda,\mu}\) has nonnegative coefficients.

By \Cref{prop:8}, the \( (i,j) \)-entry of the matrix in \eqref{eq:11}
is the generating function for the paths from \( u_{\mu_j+n-j}\) to
\(v_{\lambda_i+n-i}\) in \(G\). Since the lattice \( G \) is planar,
by the LGV lemma, we obtain the following lemma; see \Cref{fig:LGVpaths} for an example.

\begin{lem}\label{thm:1}
Let \(\lambda,\mu\in\Par_n\).
  Then we have
\[
    \widetilde N^Z_{\lambda,\mu}
    =
    \sum_{(\pi_1,\ldots,\pi_n)}
    \prod_{i=1}^n \wt(\pi_i),
  \]
  where the sum ranges over all families \((\pi_1,\ldots,\pi_n)\) of
  pairwise non-intersecting paths 
  \[
    \pi_i\in P(u_{\mu_i+n-i}\to v_{\lambda_i+n-i}), \qquad 1\le i\le n.
  \]
In particular,
\(\widetilde N^Z_{\lambda,\mu}\) is a polynomial in
\( \alpha,\beta,\gamma,\delta \) with nonnegative integer
coefficients.
\end{lem}

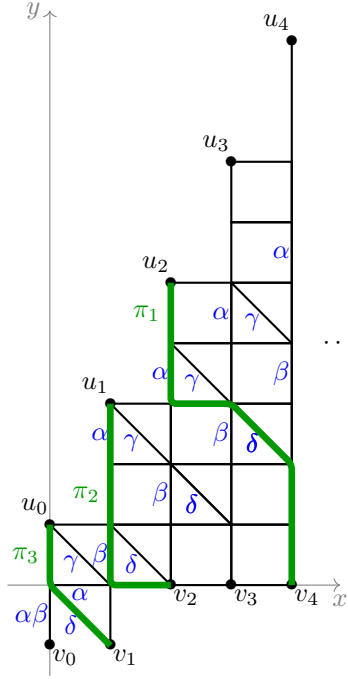
\begin{figure}
  \centering
  \begin{tikzpicture}[x=0.8cm,y=0.8cm, line width=0.8pt]
  \def\N{3}
  \def\pad{0.1}

  \pgfmathsetmacro{\Xright}{\N+1}
  \pgfmathsetmacro{\YtopLine}{2*(\N+1)+2}
  \pgfmathsetmacro{\YbotLine}{1}

  \draw[gray,thin,->] (-0.7,1) -- (\Xright+0.8,1) node[below] {$x$};
  \draw[gray,thin,->] (0,-.5) -- (0,\YtopLine+0.5) node[left] {$y$};
  \foreach \k in {0,...,\N} {
    \pgfmathsetmacro{\xL}{\N-\k}
    \pgfmathsetmacro{\yBot}{2*(\N-\k)}
    \pgfmathsetmacro{\yTop}{2*(\N-\k)+1}

    \foreach \j in {0,...,\k} {
      \pgfmathsetmacro{\x}{\xL+\j}
      \ifnum\k=\N
        \draw (\x,\yTop) rectangle ++(1,1);
      \else
        \draw (\x,\yTop) rectangle ++(1,1);
        \draw (\x,\yBot) rectangle ++(1,1);
        \draw (\k+1,\k+2) -- (\k+2,\k+1);
        \node[blue] at (\k+1.35,\k+1.35) {\( \delta \)};
      \fi
    }
    \node[blue] at (\k+0.83,\k+3.5) {\( \alpha \)};
    \node[blue] at (\k+.83,\k+1.5) {\( \beta \)};
    \node at (\k,2*\k+2) {\( \bullet \)};
    \node at (\k-0.25,2*\k+2.3) {\( u_\k \)};
    \ifnum\k<1
      \node at (0,0) {\( \bullet \)};
      \node at (0.25,-.2) {\( v_0 \)};
      \node at (1,0) {\( \bullet \)};
      \node at (1.25,-.2) {\( v_1 \)};
      \draw (0,2) -- (1,1);
      \node[blue] at (0.35,1.35) {\( \gamma \)};
    \else
      \pgfmathtruncatemacro{\K}{\k+1}
      \node at (\k+1,1) {\( \bullet \)};
      \node at (\k+1.25,0.8) {\( v_{\K} \)};
      \draw (\k,\k+3) -- (\k+1,\k+2);
      \node[blue] at (\k+0.35,\k+2.35) {\( \gamma \)};
    \fi
  }
  \node[blue] at (.5,.85) {\( \alpha \)};
  \node[blue] at (-.3,.5) {\( \alpha\beta \)};
  \node[blue] at (.35,.35) {\( \delta \)};
  
  \draw (0,1) -- (0,0);
  \draw (1,1) -- (1,0);
  \draw (0,1) -- (1,0);
  \draw (\Xright,\YbotLine) -- (\Xright,\YtopLine);
  \node[right] at (\Xright+0.35,5.0) {\(\cdots\)};

  \draw[line width=2.5pt, green!60!black, rounded corners=2pt] 
    (0,2) -- (0,1) -- (1,0);
  \node[green!60!black, font=\bfseries] at (-.4, 1.5) {\( \pi_3 \)};

  \draw[line width=2.5pt, green!60!black, rounded corners=2pt] 
    (1,4) -- (1,1) -- (2,1);
  \node[green!60!black, font=\bfseries] at (0.6, 2.5) {\( \pi_2 \)};

  \draw[line width=2.5pt, green!60!black, rounded corners=2pt ] 
    (2,6) -- (2,4) -- (3,4) -- (4,3) -- (4,1);
  \node[green!60!black, font=\bfseries] at (1.6, 5.5) {\( \pi_1 \)};

  \node at (4,10) {\( \bullet \)};
  \node at (4-0.25,10.3) {\( u_4 \)};
  \end{tikzpicture}
  \caption{An example of a family \( (\pi_1, \pi_2, \pi_3) \) of
    pairwise non-intersecting paths with \( \pi_1\in P(u_2\to v_4) \),
    \( \pi_2\in P(u_1\to v_2) \), and \( \pi_3\in P(u_0\to v_1) \). Since
    \( \wt(\pi_1)=\alpha\delta, \wt(\pi_2)=\alpha\beta, \) and
    \( \wt(\pi_3)=\delta \), this family contributes the monomial \( \alpha^2\beta\delta^2 \) to \( \widetilde{N}^Z_{(2,1,1),(0,0,0)} \).
    }
  \label{fig:LGVpaths}
\end{figure}

The following lemma presents a relation between \(N^Z_{\lambda,\mu}\) and \(\widetilde N^Z_{\lambda,\mu}\).

\begin{lem}\label{lem:8}
  Let \(\lambda,\mu\in\Par_n\) with \( n\ge1 \). 
  \begin{enumerate}
  \item If \( (\lambda_n,\mu_n)\ne(1,0) \), then
    \[
      (-1)^{|\lambda|-|\mu|}
      (\alpha\beta)^n
      N^Z_{\lambda,\mu}
      =
      \widetilde N^Z_{\lambda,\mu}.
    \]

  \item If \((\lambda_n,\mu_n)=(1,0)\) and either \(n=1\) or \(\mu_{n-1}>0\), then
    \[
      (-1)^{|\lambda|-|\mu|}
      (\alpha\beta-\gamma\delta)(\alpha\beta)^{n-1}
      N^Z_{\lambda,\mu}
      =
      \widetilde N^Z_{\lambda,\mu}.
    \]

  \item If \((\lambda_n,\mu_n)=(1,0)\), \(n\ge2\), and \(\mu_{n-1}=0\), then
    \[
      (-1)^{|\lambda|-|\mu|}
      (\alpha\beta-\gamma\delta)(\alpha\beta)^{n-1}
      N^Z_{\lambda,\mu}
      =
      \widetilde N^Z_{\lambda,\mu}
      +
      \gamma\delta\,
      \widetilde N^Z_{(\lambda_1+1,\ldots,\lambda_{n-1}+1),
        (\mu_1+1,\ldots,\mu_{n-2}+1,\mu_n)}.
    \]
  \end{enumerate}
\end{lem}

\begin{proof}
To compare the determinants \(N^Z_{\lambda,\mu}\) and
\(\widetilde N^Z_{\lambda,\mu}\) using \eqref{eq:11}, set
\(r_i=\lambda_i+n-i\) and \(s_i=\mu_i+n-i\) for \(1\le i\le n\), so
that
\begin{equation}\label{eq:23}
  N^Z_{\lambda,\mu}=\det(\nu^Z_{r_i,s_j})_{i,j=1}^n,
  \qquad
  \widetilde N^Z_{\lambda,\mu}
  =\det(\widetilde\nu^Z_{r_i,s_j})_{i,j=1}^n.
\end{equation}
Then \( \sum_{i=1}^n r_i-\sum_{i=1}^n s_i=|\lambda|-|\mu| \).

We first record a simple observation. If \((1,0)\) does not occur among
the pairs \((r_i,s_j)\), then by \eqref{eq:13} and  \eqref{eq:12}, we have
\( \widetilde{\nu}^Z_{r_i,s_j} = (-1)^{r_i-s_j}\alpha\beta\,\nu^Z_{r_i,s_j} \)
for all \(i,j\). Thus, in this case, we have 
\( \widetilde N^Z_{\lambda,\mu} =
(-1)^{|\lambda|-|\mu|}(\alpha\beta)^nN^Z_{\lambda,\mu} \).

For the first statement, suppose that \((\lambda_n,\mu_n)\ne(1,0)\). If \((1,0)\) does
not occur among the pairs \((r_i,s_j)\), then the desired identity
follows from the above observation. Suppose \((r_i,s_j)=(1,0)\) for some
\(i\) and \(j\). Then \(s_j=\mu_j+n-j=0\) forces \(j=n\) and
\(\mu_n=0\). Since \(r_i=\lambda_i+n-i=1\), we have \(i=n\)
and \( \lambda_n=1 \), or \(i=n-1\) and \( \lambda_{n-1}=0 \). Since \((\lambda_n,\mu_n)\ne (1,0)\), we must have
\(\lambda_{n-1}=0\), and thus \(\lambda_n=0\). Therefore
\(r_n=0\) and \(s_n=0\). Then, since \( \nu^Z_{r_n,s_n}=1 \) and
\( \widetilde\nu^Z_{r_n,s_n}=\alpha\beta \) are the only nonzero
entries in the last rows of the matrices in \eqref{eq:23}, we have
\begin{equation}\label{eq:24}
  N^Z_{\lambda,\mu}=\det(\nu^Z_{r_i,s_j})_{i,j=1}^{n-1},
  \qquad
  \widetilde N^Z_{\lambda,\mu}
  =\alpha \beta \det(\widetilde\nu^Z_{r_i,s_j})_{i,j=1}^{n-1}.
\end{equation}
Since \((r_i,s_j)\ne (1,0)\) for all \( 1\le i,j\le n-1 \)
and \( \lambda_n=\mu_n=0 \), the observation above gives
\[
\det(\widetilde\nu^Z_{r_i,s_j})_{i,j=1}^{n-1}
=
(-1)^{|\lambda|-|\mu|}(\alpha\beta)^{n-1}
\det(\nu^Z_{r_i,s_j})_{i,j=1}^{n-1}.
\]
Combining this with \eqref{eq:24}, we obtain the first statement.

For the second statement, suppose that \((\lambda_n,\mu_n)=(1,0)\) and
either \(n=1\) or \(\mu_{n-1}>0\). Then \(r_n=1\), \(s_n=0\), and
\( s_{n-1}\ge2 \) (if \( n\ge2 \)). Thus,
\( \nu^Z_{r_n,s_n}=\nu^Z_{1,0}= -
(\alpha+\beta+\gamma+\delta)/(\alpha\beta-\gamma\delta) \) and
\(
\widetilde\nu^Z_{r_n,s_n}=\widetilde\nu^Z_{1,0}=\alpha+\beta+\gamma+\delta
\) are the only nonzero entries in the last rows of the matrices in
\eqref{eq:23}. Therefore,
\begin{equation}\label{eq:25}
  N^Z_{\lambda,\mu}=
  - \frac{\alpha+\beta+\gamma+\delta}{\alpha\beta-\gamma\delta}
  \det(\nu^Z_{r_i,s_j})_{i,j=1}^{n-1},
  \qquad
  \widetilde N^Z_{\lambda,\mu}
  = (\alpha+\beta+\gamma+\delta) \det(\widetilde\nu^Z_{r_i,s_j})_{i,j=1}^{n-1}.
\end{equation}
Since \((r_i,s_j)\ne (1,0)\) for all \( 1\le i,j\le n-1 \)
and \((\lambda_n,\mu_n)=(1,0)\), the observation above gives
\[
\det(\widetilde\nu^Z_{r_i,s_j})_{i,j=1}^{n-1}
=
(-1)^{|\lambda|-|\mu|-1}(\alpha\beta)^{n-1}
\det(\nu^Z_{r_i,s_j})_{i,j=1}^{n-1}.
\]
Combining this with \eqref{eq:25}, we obtain the second statement.

For the last statement, suppose \((\lambda_n,\mu_n)=(1,0)\),
\(n\ge2\), and \(\mu_{n-1}=0\). Then \( r_n=1 \), \(s_{n-1}=1\), and
\(s_{n}=0\). Since
\( \nu^Z_{r_n,s_n}=\nu^Z_{1,0}= -
(\alpha+\beta+\gamma+\delta)/(\alpha\beta-\gamma\delta) \),
\( \nu^Z_{r_n,s_{n-1}}=\nu^Z_{1,1}=1 \),
\(
\widetilde\nu^Z_{r_n,s_n}=\widetilde\nu^Z_{1,0}=\alpha+\beta+\gamma+\delta
\), and
\( \widetilde\nu^Z_{r_n,s_{n-1}}=\widetilde\nu^Z_{1,1}=\alpha\beta \)
are the only nonzero entries in the last rows of the matrices in
\eqref{eq:23}, we have
\begin{align}
\label{eq:26}  N^Z_{\lambda,\mu}
  &=-\frac{\alpha+\beta+\gamma+\delta}{\alpha\beta-\gamma\delta}
    \det(\nu^Z_{r_i,s_j})_{i,j=1}^{n-1}
    -\det(\nu^Z_{r_i,s'_j})_{i,j=1}^{n-1},\\
  \label{eq:29}
  \widetilde N^Z_{\lambda,\mu}
  &=(\alpha+\beta+\gamma+\delta)
    \det(\widetilde\nu^Z_{r_i,s_j})_{i,j=1}^{n-1}
    -\alpha\beta\det(\widetilde\nu^Z_{r_i,s'_j})_{i,j=1}^{n-1},
\end{align}
where \( s'_j=s_j \) if \( 1\le j\le n-2 \)
and \( s'_{n-1} = s_{n}=0 \).

Since \((r_i,s_j)\ne(1,0)\) for all
\(1\le i,j\le n-1\), by the above observation
with \(\sum_{i=1}^{n-1}r_i-\sum_{j=1}^{n-1}s_j
=|\lambda|-|\mu|-1\),
we have
\begin{equation}\label{eq:27}
\det(\widetilde\nu^Z_{r_i,s_j})_{i,j=1}^{n-1}
=(-1)^{|\lambda|-|\mu|-1}(\alpha\beta)^{n-1}
\det(\nu^Z_{r_i,s_j})_{i,j=1}^{n-1}.
\end{equation}
Similarly, since \((r_i,s'_j)\ne(1,0)\) for all
\(1\le i,j\le n-1\), by the above observation
with
\(\sum_{i=1}^{n-1}r_i-\sum_{j=1}^{n-1}s'_j
=|\lambda|-|\mu|-r_n+s_{n-1}=|\lambda|-|\mu|\), we have
\begin{equation}\label{eq:28}
\det(\widetilde\nu^Z_{r_i,s'_j})_{i,j=1}^{n-1}
=(-1)^{|\lambda|-|\mu|}(\alpha\beta)^{n-1}
\det(\nu^Z_{r_i,s'_j})_{i,j=1}^{n-1}.
\end{equation}
By \eqref{eq:26}, we have
\begin{align*}
  (-1)^{|\lambda|-|\mu|}(\alpha\beta-\gamma\delta)(\alpha\beta)^{n-1}
  N^Z_{\lambda,\mu}
  &=
  (-1)^{|\lambda|-|\mu|-1}(\alpha+\beta+\gamma+\delta)(\alpha\beta)^{n-1}
  \det(\nu^Z_{r_i,s_j})_{i,j=1}^{n-1}\\
  &\qquad
  -(-1)^{|\lambda|-|\mu|}(\alpha\beta-\gamma\delta)(\alpha\beta)^{n-1}
  \det(\nu^Z_{r_i,s'_j})_{i,j=1}^{n-1}.
\end{align*}
By \eqref{eq:27}, \eqref{eq:28}, and \eqref{eq:29}, we
can rewrite the above equation as
\begin{align*}
  &(-1)^{|\lambda|-|\mu|}(\alpha\beta-\gamma\delta)(\alpha\beta)^{n-1}
  N^Z_{\lambda,\mu}\\
  &\quad=
  (\alpha+\beta+\gamma+\delta)\det(\widetilde\nu^Z_{r_i,s_j})_{i,j=1}^{n-1}
  -\alpha\beta\det(\widetilde\nu^Z_{r_i,s'_j})_{i,j=1}^{n-1}
  +\gamma\delta\det(\widetilde\nu^Z_{r_i,s'_j})_{i,j=1}^{n-1}\\
  &\quad=\widetilde N^Z_{\lambda,\mu}
  +\gamma\delta\det(\widetilde\nu^Z_{r_i,s'_j})_{i,j=1}^{n-1}.
\end{align*}
Finally, since
\[
\det(\widetilde\nu^Z_{r_i,s'_j})_{i,j=1}^{n-1}
=\widetilde N^Z_{(\lambda_1+1,\ldots,\lambda_{n-1}+1),
  (\mu_1+1,\ldots,\mu_{n-2}+1,\mu_n)},
\]
we obtain the last statement. This completes the proof.
\end{proof}

Recall that \( \vec1_{P} \) is \( 1 \) if the statement \( P \) is
true and \( 0 \) otherwise. By \Cref{thm:1,lem:8}, we obtain the
following result.

\begin{prop}\label{pro:1}
  Let \(\lambda,\mu\in\Par_n\) with
  \( n\ge1 \). Then the following is a polynomial with nonnegative
  integer coefficients:
  \[
    (-1)^{|\lambda|-|\mu|}(\alpha\beta -
    {\vec 1}_{(\lambda_n,\mu_n)=(1,0)}\gamma\delta)
    (\alpha\beta)^{n-1}N_{\lambda, \mu}^Z.
  \]
\end{prop}

By the duality in \Cref{lem:10}, we obtain a positivity result for
\(M^Z_{\lambda,\mu}(1;\alpha,\beta,\gamma,\delta;0)\).

\begin{prop} \label{cor:4}
  Let \( \lambda, \mu  \in \Par_n \) with
  \( \lambda_1, \mu_1 \le m \).
  Then the following
  is a polynomial with nonnegative integer coefficients:
  \[
    (\alpha\beta - {\vec 1}_{(\lambda'_1,\mu'_1)=(n,n-1)}\gamma\delta)
    (\alpha\beta)^{m-1}M_{\lambda, \mu}^Z.
  \]
\end{prop}
\begin{proof}
  Let \( \widetilde\lambda, \widetilde\mu \in \Par_m \) be the
  partitions given by \( \widetilde\lambda_j = n - \lambda'_{m+1-j} \)
  and \( \widetilde\mu_j = n - \mu'_{m+1-j} \). Then by \Cref{lem:10}, we have
  \[
    M_{\lambda, \mu}^Z = (-1)^{|\lambda|-|\mu|}N_{\widetilde\mu, \widetilde\lambda}^Z
    = (-1)^{|\widetilde{\mu}|-|\widetilde{\lambda}|}N_{\widetilde\mu, \widetilde\lambda}^Z.
  \]
  Thus it suffices to show that the following
  is a polynomial with nonnegative integer coefficients:
  \[
    (-1)^{|\widetilde{\mu}|-|\widetilde{\lambda}|}
    (\alpha\beta - {\vec 1}_{(\lambda'_1,\mu'_1)=(n,n-1)}\gamma\delta)
    (\alpha\beta)^{m-1} N_{\widetilde\mu, \widetilde\lambda}^Z.
  \]
  Since \( (\widetilde\mu_m, \widetilde\lambda_m) = (1, 0) \) if and
  only if \( (\lambda'_1, \mu'_1) = (n, n-1) \), this follows from \Cref{pro:1}.
\end{proof}

\begin{remark}
  The denominators of \( M_{\lambda, \mu}^Z \) and \( N_{\lambda, \mu}^Z \)
  in the case \( (\xi, q) = (1, 0) \)
  shown in \Cref{lem:8} and \Cref{cor:4} are also denominators of
  \( M_{\lambda, \mu}^Z \) and \( N_{\lambda, \mu}^Z \)
  in the case \( q = 0 \) with a general \( \xi \).
\end{remark}

The denominators in \Cref{lem:8} and \Cref{cor:4} are not the minimal
denominators in general.
In \Cref{prop:minimal_q0xi1,prop:minimal_q0_M}, we will determine the minimal denominators of
\( N^Z_{\lambda, \mu} \) and \( M^Z_{\lambda, \mu} \), respectively,
for the case \( (\xi, q) = (1, 0) \).

\subsection{The case \( (\xi,q)=(1,1) \)}
\label{subsec:q1-xi1}

In this subsection, we define a quantity
\(\widetilde{M}^{Z}_{\lambda,\mu}\), which will be shown to be a
numerator of \(M^{Z}_{\lambda,\mu}\) under the specialization
\((\xi,q)=(1,1)\). We then derive an explicit formula for
\(\widetilde{M}^{Z}_{\lambda,\mu}\) and obtain a positivity result for
this numerator. Throughout this subsection, we assume
\( (\xi,q)=(1,1) \). For simplicity, we write
\( M^Z_{\lambda,\mu} =
M^Z_{\lambda,\mu}(1;\alpha,\beta,\gamma,\delta;1) \).

We begin with the following known formula for
\(\widetilde{Z}_{n,k}(1;\alpha,\beta,\gamma,\delta;1)\).

\begin{prop}\cite[Proposition~3.16]{CMW}\label{prop:case2}
We have
\begin{equation}\label{eq:3}
  \widetilde{Z}_{n, k}(1;\alpha,\beta,\gamma,\delta;1)
  =\binom{n}{k} \prod_{i=k}^{n-1}\left( \alpha+\beta+\gamma+\delta+i(\alpha+\gamma)(\beta+\delta)\right) .
\end{equation}
\end{prop}

The formula \eqref{eq:3} for the specialization
\( (\gamma,\delta) = (0,0) \) was first obtained by Mandelshtam and
Viennot~\cite[Theorem~2.19]{Mandelshtam2018a}, using a recurrence for
rhombic alternative tableaux. Building on their result, Corteel et
al.~\cite[Proposition~3.16]{CMW} then derived \eqref{eq:3} using the
symmetry of rhombic staircase tableaux at \(q=1\). Below, we give a
short proof of \eqref{eq:3} using the mixed moments of Laguerre
polynomials.

The \emph{monic Laguerre polynomials} \(L_n^{(p)}(x)\) are defined by
\begin{equation}\label{eq:19}
  L_{n+1}^{(p)}(x)
  =
  (x-2n-p-1)L_n^{(p)}(x)
  -
  n(n+p)L_{n-1}^{(p)}(x),
  \qquad n\ge 0,
\end{equation}
with initial conditions \(L_{-1}^{(p)}(x)=0\) and
\(L_0^{(p)}(x)=1\).
It is well known that their mixed moments $\sigma^{(p)}_{n,k}$ are given by
\begin{equation}\label{eq:16}
  \sigma^{(p)}_{n,k} = \frac{n!}{k!}\binom{n+p}{n-k}.
\end{equation}

\begin{proof}[Proof of \Cref{prop:case2}]
From \Cref{cor:q=1} with \(\xi=1\), we have
\[
b_m^Z
=
\frac{
\alpha+\beta+\gamma+\delta
+2m(\alpha+\gamma)(\beta+\delta)
}{
\alpha\beta-\gamma\delta
},
\]
and
\[
\lambda_m^Z
=
\frac{
m(\alpha+\gamma)(\beta+\delta)
\left((m-1)(\alpha+\gamma)(\beta+\delta)
+\alpha+\beta+\gamma+\delta\right)
}{
(\alpha\beta-\gamma\delta)^2
}.
\]
By \Cref{thm:Z_nk}, we have \( Z_{n,k}(1 ; \alpha, \beta, \gamma, \delta ; 1) = \sigma_{n,k}^Z \).
Applying \Cref{lem:3} and \Cref{lem:5}, we obtain
\begin{equation}\label{eq:17}
\widetilde{Z}_{n,k}(1 ; \alpha, \beta, \gamma, \delta ; 1)
= (\alpha\beta-\gamma\delta)^{n-k}
\sigma_{n,k}^Z
= \sigma_{n,k}(\{\widetilde b_m\},\{\widetilde\lambda_m\}),
\end{equation}
where
\(\widetilde b_m = (\alpha\beta-\gamma\delta)b_m^Z \) and \( \widetilde\lambda_m = (\alpha\beta-\gamma\delta)^2\lambda_m^Z \).
For simplicity, set
\[
A = \alpha+\beta+\gamma+\delta,
\qquad
B = (\alpha+\gamma)(\beta+\delta),
\qquad
p = \frac{A}{B} - 1.
\]
Then a direct computation gives
\(\widetilde{b}_m = B(2m+p+1) \) and \( \widetilde{\lambda}_m = B^2m(m+p) \).
Comparing these with the three-term recurrence \eqref{eq:19} of the monic Laguerre polynomials, we obtain
\[
P_n(x;\{\widetilde b_m\},\{\widetilde\lambda_m\}) = B^n L_n^{(p)}(x/B).
\]
Hence, by \Cref{lem:2} and \eqref{eq:16}, we have
\begin{equation}\label{eq:18}
\sigma_{n,k}(\{\widetilde b_m\},\{\widetilde\lambda_m\})
= B^{n-k}\sigma^{(p)}_{n,k} = B^{n-k}\frac{n!}{k!}\binom{n+p}{n-k}= \binom{n}{k}\prod_{j=k}^{n-1}(A+Bj).
\end{equation}
By \eqref{eq:17} and \eqref{eq:18}, we obtain the desired formula.
\end{proof}

For \(\lambda,\mu \in \Par_n\), we define 
\[\widetilde{M}^{Z}_{\lambda,\mu}=
  \det(\widetilde{Z}_{\lambda_i+n-i, \mu_j+n-j}(1;\alpha,\beta,\gamma,\delta;1))_{i, j = 1}^n.\]
Later, we will show that \(\widetilde{M}^{Z}_{\lambda,\mu}\) is the minimal numerator of \(M^{Z}_{\lambda,\mu}(1;\alpha,\beta,\gamma,\delta;1)\).
Using the explicit formula for $\widetilde{Z}_{n,k}$, we
obtain the following expression for $\widetilde{M}^{Z}_{\lambda,\mu}$. 

\begin{prop}\label{thm:q1-explicit}
For any \( \lambda,\mu\in \Par_n \) with \( \mu\subseteq\lambda \), we have
\begin{equation*}
\widetilde{M}^{Z}_{\lambda,\mu}
=
\det\left(\binom{\lambda_i+n-i}{\mu_j+n-j}\right)_{i,j=1}^n
\prod_{z\in\lambda/\mu}
\left(
  \alpha+\beta+\gamma+\delta
  +(c(z)+n-1)(\alpha+\gamma)(\beta+\delta)
\right),
\end{equation*}
where \(c(z)\) is the content of the cell \(z\) in \(\lambda/\mu\).
\end{prop}
\begin{proof}
  Set
  \(A=\alpha+\beta+\gamma+\delta\), \(B=(\alpha+\gamma)(\beta+\delta)\),
  and
  \(F_m=\prod_{k=0}^{m-1}\left(A+kB\right)\). By the definition of
  \(\widetilde{M}^{Z}_{\lambda,\mu}\) and \eqref{eq:3}, we have
  \[
    \widetilde{M}^{Z}_{\lambda,\mu}
    =
    \det\left(
      \binom{\lambda_i+n-i}{\mu_j+n-j}
      \frac{F_{\lambda_i+n-i}}{F_{\mu_j+n-j}}
    \right)_{i,j=1}^n .
  \]
  Factoring out \(F_{\lambda_i+n-i}\) from each row \(i\) and
  \(F_{\mu_j+n-j}^{-1}\) from each column \(j\), we obtain
  \begin{align*}
    \widetilde{M}^{Z}_{\lambda,\mu}
    &=
    \det\left(\binom{\lambda_i+n-i}{\mu_j+n-j}\right)_{i,j=1}^n
    \prod_{i=1}^{n}
    \frac{F_{\lambda_i+n-i}}{F_{\mu_i+n-i}}\\
  &= \det\left(\binom{\lambda_i+n-i}{\mu_j+n-j}\right)_{i,j=1}^n
  \prod_{i=1}^{n}
  \prod_{j=\mu_i+1}^{\lambda_i}
  \left(A+(j-i+n-1)B\right).
\end{align*}
Rewriting the product in terms of the cells in \(\lambda/\mu\), we obtain the desired identity.
\end{proof}

The following proposition shows that
\(\widetilde{M}^{Z}_{\lambda,\mu}\) is a numerator of
\(M^{Z}_{\lambda,\mu}\) at the specialization \(q=1\).

\begin{prop}\label{prop:Z_multi}
For any \( \lambda,\mu\in \Par_n \),  we have
\[
  M^{Z}_{\lambda,\mu}(1;\alpha,\beta,\gamma,\delta;1)
  =\frac{\widetilde{M}^{Z}_{\lambda,\mu}}{(\alpha\beta-\gamma\delta)^{|\lambda|-|\mu|}} .
\]
\end{prop}
\begin{proof}
  By the definition of $\widetilde{M}^{Z}_{\lambda,\mu}$ and
  \Cref{lem:3}, we have
\begin{align*}
\widetilde{M}^{Z}_{\lambda, \mu}
&= \det\!\left(
  Z_{\lambda_i+n-i, \mu_j+n-j}(1;\alpha,\beta,\gamma,\delta;1)
(\alpha\beta-\gamma\delta)^{\lambda_i-\mu_j-i+j}
\right)_{i,j=1}^n\\
&= (\alpha\beta-\gamma\delta)^{|\lambda|-|\mu|}\det\!\left(
  Z_{\lambda_i+n-i, \mu_j+n-j}(1;\alpha,\beta,\gamma,\delta;1)
\right)_{i,j=1}^n.
\end{align*}
Since
\(M^{Z}_{\lambda,\mu}(1;\alpha,\beta,\gamma,\delta;1) =
\det(\sigma_{\lambda_i+n-i,\mu_j+n-j}^Z)_{i,j=1}^n \), \Cref{thm:Z_nk}
yields the desired formula.
\end{proof}

The quantity \(\widetilde{M}^{Z}_{\lambda,\mu}\) admits a
combinatorial interpretation in terms of lecture hall tableaux. An
\emph{\(n\)-lecture hall tableau} of shape \(\lambda/\mu\) with bound
\( m \) is a filling \(T\) of the cells \((i,j)\in\lambda/\mu\) with
nonnegative integers satisfying
\[
\frac{T(i,j)}{n+c(i,j)} \ge \frac{T(i,j+1)}{n+c(i,j+1)}, 
\qquad
\frac{T(i,j)}{n+c(i,j)} > \frac{T(i+1,j)}{n+c(i+1,j)},
\qquad 
0 \le \frac{T(i,j)}{n+c(i,j)} < m.
\]
We denote by \(\LHT_{n,m}(\lambda/\mu)\) the set of such tableaux. 

By \cite[Proposition~1.2]{Corteel2020} with \( m=1 \),
for \(\lambda,\mu\in\Par_n\) with \(\mu\subseteq\lambda\),
we have
\begin{equation}\label{eq:21}
\left|\LHT_{n,1}(\lambda / \mu)\right|
=
\det\left(\binom{\lambda_i+n-i}{\mu_j+n-j}\right)_{i,j=1}^{n}.
\end{equation}
By \Cref{thm:q1-explicit} and \eqref{eq:21}, we obtain the
following corollary, which gives a manifestly positive combinatorial
formula for the numerator \( \widetilde{M}^{Z}_{\lambda,\mu} \) of
\(M^{Z}_{\lambda,\mu}\) when \((\xi,q)=(1,1)\).

\begin{cor}\label{cor:1}
For \(\lambda,\mu\in\Par_n\) with \(\mu\subseteq\lambda\), we have
\[
\widetilde{M}^{Z}_{\lambda,\mu}
=
\left|\LHT_{n,1}(\lambda/\mu)\right|
\prod_{z\in\lambda/\mu}
\left(
  \alpha+\beta+\gamma+\delta
  +(c(z)+n-1)(\alpha+\gamma)(\beta+\delta)
\right).
\]
In particular,
the following is a
polynomial in \(\alpha,\beta,\gamma,\delta\) with nonnegative integer
coefficients.
\[
\widetilde{M}^{Z}_{\lambda,\mu}  
=(\alpha\beta-\gamma\delta)^{|\lambda|-|\mu|}
M^{Z}_{\lambda,\mu}(1;\alpha,\beta,\gamma,\delta;1) .
\]
\end{cor}

The explicit formula for \(\widetilde{M}^{Z}_{\lambda,\mu}\) also
recovers the following known product formula for the Koornwinder
moments
\(M^{Z}_{\lambda}=M^{Z}_{\lambda,\emptyset}(1;\alpha,\beta,\gamma,\delta;1)\).
We include a proof to show how it follows from the specialization
\(\mu=\emptyset\) of our formula.

\begin{cor} \cite[Theorem~8.4]{Corteel2019}
  \label{cor:CW-prod}
  Suppose \( (\xi,q) = (1,1) \).
For a partition \( \lambda \in \Par_n \), we have
\[
 M^{Z}_\lambda
=S^{|\lambda|} \prod_{z \in \lambda}(D+h(z)-1) \cdot \prod_{i=1}^{n-1} \prod_{j=i+1}^n \frac{\left(D+\lambda_i-\lambda_j+j-i-1\right)\left(\lambda_i-\lambda_j+j-i\right)}{(D+j-i-1)(j-i)} ,
\]
where \( h(z) \) is the hook length and
\[
D=\frac{\alpha+\beta+\gamma+\delta}{(\alpha+\gamma)(\beta+\delta)} , \qquad 
S=\frac{(\alpha+\gamma)(\beta+\delta)}{\alpha \beta-\gamma \delta} .
\]
\end{cor}
\begin{proof}
By \cite[Corollary~1.3]{Corteel2020} with \( m=1 \), we have
\[
  \left|\LHT_{n,1}(\lambda)\right| =\prod_{z\in\lambda} \frac{n+c(z)}{h(z)}.
\]
Hence, by \Cref{prop:Z_multi} and \Cref{cor:1}, we obtain
\[
M^{Z}_\lambda
=
S^{|\lambda|}
\prod_{z\in\lambda}
\frac{n+c(z)}{h(z)}
\left(D+c(z)+n-1\right).
\]
Using the diagrammatic argument in the proofs of Lemma~7.21.1 and
Theorem~7.21.2 of \cite{Stanley1999}, we obtain the identities
\[
\prod_{z\in\lambda}\frac{n+c(z)}{h(z)}
=
\prod_{1\le i<j\le n}
\frac{\lambda_i-\lambda_j+j-i}{j-i},
\]
and
\[
\prod_{z\in\lambda}\frac{D+n+c(z)-1}{D+h(z)-1}
=
\prod_{1\le i<j\le n}
\frac{D+\lambda_i-\lambda_j+j-i-1}{D+j-i-1}.
\]
Combining these identities completes the proof.
\end{proof}

\section{Minimal Denominators}
\label{sec:minimal-denominators}

In the previous sections, we obtained denominators for
\(M^Z_{\lambda,\mu}\) in the cases \(\mu=\emptyset\), \((\xi,q)=(1,0)\),
and \((\xi,q)=(1,1)\), together with positive numerators in the latter
two specializations. In this section, we first prove a general lemma showing that, under
a \((0,1)\)-specialization, any such positive numerator yields the
positivity of the minimal numerator. This establishes the generalized Rains'
conjecture for \((\xi,q)=(1,0)\) and \((\xi,q)=(1,1)\).
We then prove the minimality of the denominators
obtained earlier in the cases \(\mu=\emptyset\) and \((\xi,q)=(1,1)\),
and determine the minimal denominator in the case \((\xi,q)=(1,0)\).

We begin with the following lemma. Recall that, by definition, every
\( (0,1) \)-substitution map \( \varphi \) satisfies
\( \varphi(\alpha\beta-q^k \gamma\delta)\ne 0 \) for all \( k\ge0 \).

\begin{lem}\label{lem:7}
  Let \( k\ge0 \) be an integer and let \( \varphi \) be a
  \( (0,1) \)-substitution map satisfying
  \( \varphi(\alpha\beta)\ne 0 \) and
  \( \varphi(q^k\gamma\delta)\ne 0 \). Let \( p \) be a rational
  function such that \( \varphi(p) \) is a nonzero polynomial in
  \( \xi,\alpha,\beta,\gamma,\delta \), and \( q \) with real
  coefficients. Then, the polynomial
  \( \varphi((\alpha\beta-q^k \gamma\delta) p) \) has a negative
  coefficient.
\end{lem}

\begin{proof}
  Suppose that the coefficients of
  \( \varphi((\alpha\beta-q^k \gamma\delta) p) \)
   are all nonnegative.
   Since \( \varphi(\alpha\beta)\ne 0 \) and \( \varphi(q^k\gamma\delta)\ne 0 \), we have
   \[
     \varphi((\alpha\beta-q^k \gamma\delta)p)|_{(\alpha,\beta,\gamma,\delta,q)=(1,1,1,1,1)}
     = 0.
   \]
   Then, since \( \varphi((\alpha\beta-q^k \gamma\delta) p) \) is a
   polynomial with nonnegative coefficients, it must be the zero
   polynomial. Since \( \varphi(\alpha\beta-q^k \gamma\delta)\ne0 \),
   we obtain \( \varphi(p)=0 \), which is a contradiction. Thus
   \( \varphi((\alpha\beta-q^k \gamma\delta) p) \) must have a negative
   coefficient.
\end{proof}

Now we are ready to prove \Cref{lem:9}.
\begin{proof}[Proof of \Cref{lem:9}]
  The ``only if'' part is obvious. For the ``if'' part, suppose that
  \( \varphi(D M_{\lambda, \mu}^Z) \) is a polynomial with
  nonnegative integer coefficients, where
  \( D = \prod_{i \in I}(\alpha\beta - q^i\gamma\delta) \) and \( I \)
  is a multiset of nonnegative integers. Then the minimal numerator
  \( \varphi(D_{\min} M_{\lambda, \mu}^Z) \) of \( \varphi(M_{\lambda, \mu}^Z) \) is
  given by \( \varphi((D/R) M_{\lambda, \mu}^Z) \),
  where \( R = \prod_{i \in J}(\alpha\beta - q^i\gamma\delta) \) for
  some multiset \( J \) of nonnegative integers.

  We claim that \( \varphi(q^k\gamma\delta)=0 \) for all
  \( k \in J \). To see this, suppose
  \( \varphi(q^k\gamma\delta) \neq 0 \) for some \( k \in J \). Let
  \[
    p = \frac{D}{\alpha\beta - q^k\gamma\delta}  \cdot M_{\lambda, \mu}^Z.
  \]
Then
\[
  \varphi(p) = \varphi \left( \frac{R}{\alpha\beta - q^k\gamma\delta} \cdot \frac{D}{R} \cdot
  M_{\lambda, \mu}^Z \right)
= \varphi \left( \frac{R}{\alpha\beta - q^k\gamma\delta}   \right)
\varphi \left( \frac{D}{R} M_{\lambda, \mu}^Z \right)
  \]
  is a polynomial. Since \( \varphi(\alpha\beta) \neq 0 \) and
  \( \varphi(q^k\gamma\delta) \neq 0 \), by \Cref{lem:7},
  \( \varphi(D M_{\lambda, \mu}^Z) = \varphi((\alpha\beta -
  q^k\gamma\delta)p ) \) has a negative coefficient, which is a
  contradiction. Therefore, the claim holds.

  By the claim, we have \( \varphi(R)=\varphi((\alpha\beta)^{|J|}) \),
  which is a monomial. Hence,
  \( \varphi(D_{\min} M_{\lambda, \mu}^Z) = \varphi(D M_{\lambda,
    \mu}^Z)/\varphi((\alpha\beta)^{|J|}) \) is a polynomial with
  nonnegative integer coefficients.
\end{proof}

Using \Cref{lem:9}, we prove the generalized Rains' conjecture under the two
specializations \( (\xi, q) = (1, 0) \) and \( (\xi, q) = (1, 1) \) as
follows.

\begin{thm}\label{thm:minimal-q0}
For \(\lambda,\mu\in\Par_n\), the minimal numerator of \(M^Z_{\lambda,\mu}(1;\alpha,\beta,\gamma,\delta;0)\) is a polynomial with
nonnegative integer coefficients.
\end{thm}
\begin{proof}
  Let \(\varphi\) be the \( (0,1) \)-substitution map defined by
  \((\xi,q)\mapsto(1,0)\). By \Cref{cor:4}, there is a polynomial
  \(D=\prod_{i\in I}(\alpha\beta-q^i\gamma\delta)\), where \( I \) is
  a multiset of nonnegative integers, such that
  \(\varphi(DM^Z_{\lambda,\mu})\) is a polynomial with nonnegative
  integer coefficients. Since \(\varphi(\alpha\beta)\neq0\), by
  \Cref{lem:9}, the minimal numerator of
  \(\varphi(M^Z_{\lambda,\mu}) =
  M^Z_{\lambda,\mu}(1;\alpha,\beta,\gamma,\delta;0)\) also has
  nonnegative integer coefficients.
\end{proof}

By the same argument with \Cref{cor:1}, we obtain the following theorem.

\begin{thm}\label{thm:2}
  For \( \lambda,\mu\in\Par_n \) with \( \mu\subseteq\lambda \),
  the minimal numerator of \( M^Z_{\lambda,\mu}(1;\alpha,\beta,\gamma,\delta;1) \) is a polynomial with nonnegative integer coefficients.
\end{thm}

Finally, we present explicit formulas for the minimal denominators of
\( M_{\lambda, \mu}^Z \) in three cases:
\( \mu = \emptyset \), \(  (\xi, q) = (1, 0) \), and \( (\xi, q) = (1, 1) \).
For the first case, the following proposition shows that the denominator of \( M_{\lambda}^Z \)
in \Cref{prop:deno} is the minimal denominator.

\begin{prop}\label{prop:den}
For \( \lambda \in \Par_n \), the following is the minimal denominator
of \(M^{Z}_\lambda\):
\begin{equation}\label{eq:7}
  \prod_{(i,j)\in\lambda}
  \left(\alpha\beta-q^{2n-2-i+j}\gamma\delta\right).
\end{equation}
\end{prop}
\begin{proof}
  Let \(f\) be the polynomial in \eqref{eq:7}. By \Cref{prop:deno},
  \(f\) is a denominator of \(M^{Z}_\lambda\). To prove the minimality
  of \( f \), suppose that
  \( p:= M^{Z}_\lambda f/\left(\alpha\beta-q^k\gamma\delta\right)\) is a
  polynomial for some factor \(\alpha\beta-q^k\gamma\delta\) of \(f\).
  Let \(\varphi\) be the \( (0,1) \)-substitution map defined by
  \((\xi,q)\mapsto(1,1)\).
  Then
  \[
   \varphi(p)
    = M^{Z}_\lambda(1;\alpha,\beta,\gamma,\delta;1) (\alpha\beta-\gamma\delta)^{|\lambda|-1}
  \]
  is a polynomial. By \Cref{prop:Z_multi}, we have
\( \widetilde{M}^{Z}_{\lambda,\emptyset} = \varphi((\alpha\beta-\gamma\delta)p) \).
Then by \Cref{lem:7} with \(k=0\),
we obtain that
\(\widetilde{M}^{Z}_{\lambda,\emptyset}\) has a negative
coefficient, which contradicts \Cref{cor:1}. Therefore \(f\) is the
minimal denominator of \(M^{Z}_\lambda\).
\end{proof}

For the case \( (\xi, q) = (1, 0) \), we first determine the denominator
of \(N^Z_{\lambda,\mu}\).

\begin{prop}\label{prop:minimal_q0xi1}
  For \( \lambda,\mu\in\Par_n \) with \( \mu\subseteq\lambda \),
  suppose that \( \lambda / \mu \) has \( m \) nonempty rows. Then the
  following is the minimal denominator of
  \( N^Z_{\lambda, \mu}(1; \alpha, \beta, \gamma, \delta; 0) \):
  \[
    (\alpha\beta - {\vec 1}_{(\lambda_n,\mu_n)=(1,0)}\gamma\delta)
    (\alpha\beta)^{m-1}.
  \]
\end{prop}
\begin{proof}
Write \( N^Z_{\lambda,\mu}=N^Z_{\lambda, \mu}(1; \alpha, \beta, \gamma, \delta; 0) \) and let
\[
    A_{\lambda,\mu} =
    (\alpha\beta - {\vec 1}_{(\lambda_n,\mu_n)=(1,0)}\gamma\delta)
    (\alpha\beta)^{m-1} N^Z_{\lambda,\mu}.
\]
Then by~\Cref{lem:8} and \Cref{thm:1}, we have 
\begin{equation}\label{eq:22}
  (-1)^{|\lambda|-|\mu|}(\alpha\beta)^{n-m}A_{\lambda,\mu} = \widetilde N^Z_{\lambda,\mu} + T
\end{equation}
where \( T=0 \) or \( T \) is a polynomial with nonnegative coefficients.
Hence, \( (-1)^{|\lambda|-|\mu|}(\alpha\beta)^{n-m}A_{\lambda,\mu} \) is a numerator of
\( N^Z_{\lambda, \mu} \).

We first show that \( A_{\lambda,\mu} \) is a polynomial using \eqref{eq:22}.
Let \(r_i=\lambda_i+n-i\) and \( s_i=\mu_i+n-i \) for \( 1\le i \le n \).
By \Cref{thm:1}, \( \widetilde N^Z_{\lambda,\mu} \)
is the generating function for non-intersecting paths
\( (\pi_1,\ldots,\pi_n) \) with \( \pi_i\in P(u_{s_i}\to v_{r_i}) \).
 For each \( i \) with \(\lambda_i=\mu_i\),
since \(r_i=s_i\), the path from \(u_{s_i}\) to \(v_{r_i}\) is unique
and has weight \(\alpha\beta\); see \Cref{fig:HJlattice}. Since
\(\lambda / \mu \) has \( n-m \) empty rows, every family of non-intersecting paths contributing to
\( \widetilde N^Z_{\lambda,\mu} \)
has weight divisible by \((\alpha\beta)^{n-m}\), so
\( \widetilde N^Z_{\lambda,\mu} \) is divisible by \((\alpha\beta)^{n-m}\).
If \( T\ne0 \), then \( T=\gamma\delta\,\widetilde
N^Z_{(\lambda_1+1,\ldots,\lambda_{n-1}+1),(\mu_1+1,\ldots,\mu_{n-2}+1,\mu_n)} \)
by \Cref{lem:8}, and the \( n-m \) empty rows of \( \lambda/\mu \) (which
all have index at most \( n-2 \) in this case) remain empty rows of
the shifted skew shape; hence \( T \) is also divisible by
\((\alpha\beta)^{n-m}\).
By \eqref{eq:22}, 
it follows that \(A_{\lambda,\mu}\) is a polynomial, hence a numerator of \( N^Z_{\lambda, \mu} \).

It remains to prove that \( A_{\lambda,\mu} \) is the minimal
numerator of \( N^Z_{\lambda, \mu} \). By the definition of \( A_{\lambda,\mu} \), it suffices to show that neither \(\alpha\beta\) nor
\(\alpha\beta-\gamma\delta\) divides \(A_{\lambda,\mu}\), since the possible factors of the minimal denominator are \( \alpha\beta \) and \( \alpha\beta-\gamma\delta \).

If \(m=0\), then \(\lambda=\mu\), the displayed denominator is \(1\),
and \(N^Z_{\lambda,\lambda}=1\), so the result follows immediately.
Thus we may assume \(m>0\). Then there exists an index \(i\) with
\(\lambda_i>\mu_i\).
For each \( i \) with \( \lambda_i>\mu_i \), we have \(r_i>s_i\).
If \(s_i=0\), choose the path \( DE^{r_i-1} \) from \(u_0\) to \(v_{r_i}\) when \(r_i\ge 2\), and the path \(DS\) from \(u_0\) to \(v_1\) when \(r_i=1\); in either case the weight is \( \gamma \).
If \(s_i\ge1\), choose the path
\(S^{s_i-1}DE^{r_i-s_i-1}S^{s_i+1} \)
from \(u_{s_i}\) to \(v_{r_i}\). Its weight is either \(\gamma\) or
\(\beta\gamma\), and in particular is not divisible by \(\alpha\).
These choices form a
family of non-intersecting paths. Therefore,
\( \widetilde{N}^Z_{\lambda, \mu} \)
has a monomial that is not divisible by \(\alpha^{n-m+1}\).
By \eqref{eq:22}, \( A_{\lambda,\mu}\) is not divisible by \( \alpha\beta \).

By \eqref{eq:22},
\( (-1)^{|\lambda|-|\mu|}(\alpha\beta)^{n-m}A_{\lambda,\mu} \) is a
nonzero polynomial with nonnegative coefficients, hence, by
\Cref{lem:7}, it is not divisible by \(\alpha\beta-\gamma\delta\).
Thus, \(\alpha\beta-\gamma\delta\) does not divide
\( A_{\lambda,\mu} \). This completes the proof.
\end{proof}

Now, we obtain the minimal denominator of
\( M^Z_{\lambda,\mu}(1;\alpha,\beta,\gamma,\delta;0) \).

\begin{prop}\label{prop:minimal_q0_M}
  For \( \lambda,\mu\in\Par_n \) with \( \mu\subseteq\lambda \), let \( k \) denote the number of indices \( i \) with \( \lambda'_i>\mu'_i \). Then the minimal denominator of \( M^Z_{\lambda,\mu}(1;\alpha,\beta,\gamma,\delta;0) \) is
  \[
    (\alpha\beta - {\vec 1}_{(\lambda'_1,\mu'_1)=(n,n-1)}\gamma\delta)(\alpha\beta)^{k-1}.
  \]
\end{prop}
\begin{proof}
  Let \( m=\lambda_1 \), and let
  \( \widetilde\lambda,\widetilde\mu\in\Par_m \) be the partitions
  given by \( \widetilde\lambda_j = n - \lambda'_{m+1-j} \) and
  \( \widetilde\mu_j = n - \mu'_{m+1-j} \). By \Cref{lem:10},
  \[
    M^Z_{\lambda,\mu}(1;\alpha,\beta,\gamma,\delta;0)
    = (-1)^{|\lambda|-|\mu|}\,N^Z_{\widetilde\mu,\widetilde\lambda}(1;\alpha,\beta,\gamma,\delta;0),
  \]
  so the minimal denominator of
  \( M^Z_{\lambda,\mu}(1;\alpha,\beta,\gamma,\delta;0) \) equals that
  of
  \(
  N^Z_{\widetilde\mu,\widetilde\lambda}(1;\alpha,\beta,\gamma,\delta;0)
  \). Since
  \( \widetilde\mu_j-\widetilde\lambda_j =
  \lambda'_{m+1-j}-\mu'_{m+1-j} \), the number of nonempty rows of
  \( \widetilde\mu/\widetilde\lambda \) equals \( k \). Since
  \( \widetilde\mu_m=n-\mu'_1 \) and
  \( \widetilde\lambda_m=n-\lambda'_1 \), we have
  \( (\widetilde\mu_m,\widetilde\lambda_m)=(1,0) \) if and only if
  \( (\lambda'_1,\mu'_1)=(n,n-1) \). Hence, the result follows from
  \Cref{prop:minimal_q0xi1}.
\end{proof}

Lastly, for \(  (\xi, q) = (1, 1) \), the following shows the minimality of the denominator in \Cref{prop:Z_multi}.
\begin{prop}\label{pro:minimalq1}
  For \( \lambda,\mu\in \Par_n \) with \( \mu\subseteq\lambda \), the
  minimal denominator of
  \( M^{Z}_{\lambda,\mu}(1;\alpha,\beta,\gamma,\delta;1) \) is
  \( (\alpha\beta-\gamma\delta)^{|\lambda|-|\mu|} \).
\end{prop}

\begin{proof}
  By \Cref{cor:1},
  \( (\alpha\beta-\gamma\delta)^{|\lambda| - |\mu|} \) is a denominator
  of \( M^{Z}_{\lambda,\mu}(1;\alpha,\beta,\gamma,\delta;1) \). By
  \Cref{prop:Z_multi},
  \( (\alpha\beta-\gamma\delta)^{|\lambda| - |\mu|}
  {M}^{Z}_{\lambda,\mu}(1;\alpha,\beta,\gamma,\delta;1) \) is a
  polynomial with nonnegative coefficients. Hence, by \Cref{lem:7},
  \( \widetilde{M}^{Z}_{\lambda,\mu} \) is not divisible by
  \( (\alpha\beta-\gamma\delta) \). Therefore,
  \( (\alpha\beta-\gamma\delta)^{|\lambda| - |\mu|} \) is the minimal
  denominator of
  \( {M}^{Z}_{\lambda,\mu}(1;\alpha,\beta,\gamma,\delta;1) \).
\end{proof}

\begin{remark}
  For \( q=1 \) with an arbitrary \( \xi \), by \Cref{lem:3}, we
  obtain that \( (\alpha\beta-\gamma\delta)^{|\lambda| - |\mu|} \) is
  a denominator of
  \( {M}^{Z}_{\lambda,\mu}(\xi;\alpha,\beta,\gamma,\delta;1) \). By
  the same argument in the proof of \Cref{pro:minimalq1}, one can show
  that it is the minimal denominator.
\end{remark}

\bibliographystyle{abbrv}

\end{document}